\documentclass[12pt]{amsart}
\usepackage{amscd}
\usepackage{amssymb}
\usepackage{a4wide}
\usepackage{amstext}
\usepackage{amsthm}
\usepackage{mathrsfs}
\usepackage{relsize}
\usepackage[final]{hyperref}
\hypersetup{unicode= false, colorlinks=true, linkcolor=blue,
anchorcolor=blue, citecolor=green, filecolor=red, menucolor=blue,
pagecolor=blue, urlcolor=blue} 
\numberwithin{equation}{section}

\newcommand{\CC}{\mathbb {C}}

\newcommand{\RR}{\mathbb{R}}

\newcommand{\Ker}{\operatorname{Ker}}

\DeclareMathOperator{\ospan}{\overline{Span}}

\renewcommand{\phi}{\varphi}

\newcommand{\ima}{{\rm Im}\,}

\newcommand{\vep}{\varepsilon}
\newcommand{\co}{\mathbb{C}}

\newcommand{\ZZ}{\mathcal{Z}}
\newcommand{\cp}{\mathbb{C^+}}
\newcommand{\cm}{\mathbb{C^-}}

\newtheorem{Thm}{Theorem}[section]
\newtheorem{theorem}[Thm]{Theorem}
\newtheorem{lemma}[Thm]{Lemma}
\newtheorem{proposition}[Thm]{Proposition}

\newtheorem{remark}[Thm]{Remark}
\newtheorem{example}[Thm]{Example}
\newcommand{\A}{\mathcal{A}}
\newcommand{\LL}{\mathcal{L}}
\newcommand{\UU}{\mathcal{U}}
\newcommand{\MM}{\mathcal{M}}
\newcommand{\EE}{\mathcal{E}}

\begin{document}
\sloppy
\title[Spectral theory of rank one perturbations of normal operators]
{Spectral theory of rank one perturbations \\ of normal compact operators}
\author{Anton Baranov}

\address{Anton Baranov,
\newline Department of Mathematics and Mechanics, St.~Petersburg State University, St.~Petersburg, Russia,
\newline National Research University Higher School of Economics, St.~Petersburg, Russia,
\newline {\tt anton.d.baranov@gmail.com}
}
\thanks{Theorems 2.1--2.6 and the results of Sections 3--6 were
obtained with the support of Russian Science Foundation project 14-21-00035. 
Theorems 2.7--2.8 and the results of Sections 7--8
were obtained as a part of joint grant of Russian Foundation for Basic Research 
(project 17-51-150005-NCNI-a) and CNRS, France 
(project PRC CNRS/RFBR 2017-2019 ``Noyaux reproduisants dans
des espaces de Hilbert de fonctions analytiques'').}

\maketitle

\begin{abstract}
We construct a functional model for rank one perturbations 
of compact normal operators acting in a certain Hilbert spaces 
of entire functions generalizing de Branges spaces. Using this model
we study completeness and spectral synthesis problems for such perturbations. 
Previously, in \cite{by}
the spectral theory of rank one perturbations was developed
in the selfadjoint case. In the present paper we extend and significantly simplify 
most of known results in the area.
We also prove an Ordering Theorem for invariant subspaces 
with common spectral part. This result is essentially new even 
for rank one perturbations of compact selfadjoint operators.
\end{abstract}

\section{Introduction}
\label{int}

\subsection{Spectral synthesis problem.}
One of the basic question in the abstract operator 
theory is whether a linear operator $\LL$ from a given class has a 
complete set of eigenvectors or {\it generalized eigenvectors} (that is, 
elements of $\Ker\,(\LL-\lambda I)^n$ for some 
$\lambda\in\mathbb{C}$ and $n\in\mathbb{N}$).
If the answer is positive, then the next question arises, whether
it is possible to reconstruct all $\LL$-invariant subspaces 
from its generalized eigenvectors. Namely, given an $\LL$-invariant subspace
$\MM$, the question is whether
\begin{equation}
\label{syn}
\MM = \ospan \Big\{x\in \MM: x\in \bigcup\limits_{\lambda, n} \Ker\, 
(\LL-\lambda I)^n\Big\},
\end{equation} 
i.e., whether $\MM$ coincides with the closed linear span of 
the generalized eigenvectors it contains.
All subspaces are assumed to be closed.

A continuous linear operator $\LL$ in  a separable Hilbert (or Banach, or Frech\'et) 
space $H$ is said to admit {\it spectral synthesis} 
if for any invariant subspace $\MM$ of $\LL$ we have \eqref{syn}.
The notion of the spectral synthesis for a general operator 
goes back to J. Wermer \cite{wer}, who 
showed, in particular, that any compact normal operator in a Hilbert 
space admits spectral synthesis.  Moreover, Wermer proved that a normal
operator $\A$ with simple eigenvalues $\lambda_n$ does not admit 
spectral synthesis if and only if the set $\{\lambda_n\}$ 
carries a complex measure orthogonal to polynomials, i.e., there exists a 
nontrivial sequence 
$\{\mu_n\} \in \ell^1$ such that $\sum_n \mu_n \lambda_n^k = 0$, $k\in \mathbb{N}_0$. 
Existence of such measures follows from classical Wolff's example \cite{wolff} 
of a Cauchy transform 
vanishing outside of the disc: there exist $\lambda_n \in  \mathbb{D} = 
\{z\in\mathbb{C}: |z|<1\}$ and $\{\mu_n\}\in\ell^1$ such that 
$\sum_{n}\frac{\mu_n}{z-\lambda_n} \equiv 0$, $|z|>1$.

The first example of a compact operator which does not admit spectral synthesis 
was implicitly given by H. Hamburger \cite{hamb} (even before Wermer's paper).
Further results were obtained  in the 1970s by N. Nikolski \cite{nik69}
and A. Markus \cite{markus}. E.g., Nikolski \cite{nik69} proved that
any Volterra operator can be a part of a complete compact operator
(recall that {\it Volterra operator} is a compact operator whose spectrum is 
$\{0\}$).


\subsection{Rank one perturbations.}
One of the major subareas of spectral theory deals with ``small''
perturbations of ``good'' operators (e.g., trace class or Schatten class perturbations
of selfadjoint operators). However, even in this case there are few general results
about completeness and synthesis. The best studied are the cases of dissipative 
operators and of {\it weak perturbations} in the sense of Keldysh and Matsaev 
(for a survey of these results see 
\cite{Gohb_Krein} or \cite{by} and references therein). 

We study spectral properties of rank one perturbations
of compact normal operators. While compact normal operators are 
among the simplest infinite-dimensional operators (being unitarily
equivalent to a diagonal operator in $\ell^2$), we will see that 
the spectral theory of their rank one perturbations is 
highly nontrivial. 

Let $\A$ be a bounded cyclic normal operator in a Hilbert space $H$. Then, by the 
Spectral Theorem, $\A$ is unitarily equivalent to the operator of multiplication
by $z$ in $L^2(\nu)$ for some finite compactly supported positive 
Borel measure $\nu$. 
In what follows we will always identify $H$ with $L^2(\nu)$ and $\A$ with multiplication 
by the variable $z$. 

For $a, b\in H$ consider the rank one perturbation $\LL
= \A + a\otimes b$ of $\A$, 
$$
 \LL x = \A x + (x,b)a, \qquad x\in H.
$$
The goal of the present paper is to study the spectral properties of 
rank one perturbations in the case when $\A$ is compact. 
In particular, we are interested 
in completeness of (generalized) eigenvectors of $\LL$ 
(in which case we say that $\LL$ {\it is complete}), 
relations between completeness of $\LL$
and its adjoint $\LL^*$, and the spectral synthesis for $\LL$. 
Earlier, in \cite{by} we considered rank one perturbations 
of selfadjoint operators. Here a unified treatment for the normal operators case
will be presented. Not only will we extend the results of \cite{bbb, by} 
to the case of normal operators, 
but also give substantially simplified proofs of several results 
from \cite{bbb, by}.

One of the main novel features of this paper is the proof of the fact
that, for a rank one perturbation, the lattice of 
all its invariant subspaces with fixed spectral part is totally ordered. 
This result is essentially new even for perturbations of selfadjoint 
operators. For the case of normal operators we will have to impose some 
conditions on the spectrum of the unperturbed operator. 
\medskip

\subsection{Notations.}
In what follows we write $U(x)\lesssim V(x)$ if 
there is a constant $C$ such that $U(x)\leq CV(x)$ holds for all $x$ 
in the set in question. We write $U(x)\asymp V(x)$ if both $U(x)\lesssim V(x)$ and
$V(x)\lesssim U(x)$. The standard Landau notations
$O$ and $o$ also will be used.

For an entire function $f$ we denote by $f^*$ the conjugate entire function
$f^*(z) = \overline{f(\bar z)}$.
The zero set of an entire function $f$ (ignoring multiplicities) 
will be denoted by $\mathcal{Z}_f$. We denote by $D(z,R)$ the disc with 
center $z$ of radius $R$. By $\mathcal{P}_n$ we denote the set of 
all polynomials of degree at most $n$. 
By $\mathbb{N}_0$ we denote the set of $n\in \mathbb{Z}$ such that $n\ge 0$.
\medskip

\subsection{Acknowledgements.} This paper is a continuation
of our work with Dmitry Yakubovich on spectral theory 
of rank one perturbations of selfadjoint operators which we started 10 years ago. 
I am deeply grateful to Dmitry for introducing me to the 
field of functional models and for our continued fruitfull colaboration. 
His influence on the subject and the results of the present paper is very important.
I also want to thank my friends and coauthors Evgeny Abakumov,
Yurii Belov and Alexander Borichev for their numerous ideas and insights 
which are widely used in this paper. And I want to express my deep gratitude
to my Teachers -- the late Professor Victor Havin and Professor 
Nikolai Nikolski -- who guided me through the exciting world of Spectral 
Function Theory.
\bigskip


\section{Main results.} 
In what follows we use the following notation. 
Given a bounded linear operator $\mathcal{U}$ we denote by 
$\EE(\UU)$ the subspace
generated by all generalized eigenvectors (also known as {\it root vectors}) 
of $\UU$. Similarly, 
given a $\UU$-invariant subspace $\MM$, we denote by $\EE(\MM, \UU)$ its 
{\it spectral part}
$$
\EE(\MM, \UU) = \ospan \Big\{x\in \MM: x\in \bigcup\limits_{\lambda, n} \Ker\, 
(\UU-\lambda I)^n\Big\}.
$$
With this notation, 
the spectral synthesis problem for $\UU$ is whether $\MM = \EE(\MM, \UU)$
for any $\UU$-invariant subspace $\MM$. 

Now we are able to state the main results of the paper. 
As we will see, there are two (apparently different) reasons 
for completeness or spectral synthesis for rank one perturbations 
(sometimes understood up to finite-dimensional complement): {\it nonvanishing moments}
and {\it domination}.

Throughout this section we always assume that 
$\A$ is a cyclic compact normal operator, i.e., 
the operator of multiplication by $z$ in $H=L^2(\nu)$, 
where $\nu = \sum_n \nu_n \delta_{s_n}$, $s_n\in \co$, $s_n \ne 0$. 

We will impose one more restriction on $\A$: 
$$
\text{there exists}\ \ p>0\ \  \text{such that} \ \ \ 
\sum_n |s_n|^p <\infty,
$$
that is, $\A$ belongs to some Schatten ideal 
$\mathfrak{S}_p$. Equivalently, this means that the sequence 
$t_n = s_n^{-1}$ 
has a finite convergence exponent and, thus, is 
the zero set of some entire function of finite order.
All main theorems except Theorems \ref{bio2}, \ref{coun} and 
\ref{ord1} apply only to this case.

We identify the elements of $H$ with sequences. For 
$a = (a_n), b = (b_n) \in H$ we consider the associated
rank one perturbation of $\A$, 
$$
\LL = \A + a\otimes b.
$$

We write $a\in z L^2(\nu)$, if $a\in \A H$, that is, 
$\sum_n |a_n|^2 |s_n|^{-2}\nu_n =\sum_n |a_n|^2 |t_n|^2\nu_n  <\infty$. 
However, as we will see
many results will depend on the products $a_n \bar b_n$ and so 
it will be convenient to introduce the following notation:
we write $a\in ' z L^2(\nu)$ if 
$\sum_{n: b_n \ne 0} |a_n|^2 |t_n|^2 \nu_n < \infty$,
and $b\in ' z L^2(\nu)$ if 
$\sum_{n: a_n \ne 0} |b_n|^2 |t_n|^2 \nu_n < \infty$.
\medskip


\subsection{Completeness of the operator and its adjoint.}
The first result shows that $\LL$ and $\LL^*$ are (nearly) complete
when certain moments of the sequence 
$a_n \bar b_n \nu_n$ do not vanish.

\begin{theorem}
\label{bio1}
Let $\A$ be a normal operator in the class $\mathfrak{S}_p$ 
and let $\LL = \A +a\otimes b$ be its rank one perturbation. 
\medskip

1. Assume that $\sum_n |a_n b_n t_n| \nu_n<\infty$ and 
\begin{equation}
\label{mom0}
\sum_n  a_n \bar b_n t_n  \nu_n \ne -1.
\end{equation}
Then both $\LL$ and $\LL^*$ are complete.
\medskip

2. Assume that there exists $N\in \mathbb{N}$ such that 
$\sum_n |a_n b_n| |t_n|^{N+1} \nu_n<\infty$, 
\begin{equation}
\label{mom1}
\begin{aligned}
& \sum_n  a_n \bar b_n t_n  \nu_n  = -1, \\
& \sum_n a_n \bar b_n t_n^k \nu_n  = 0, \quad k=2, \dots N, \\
& \sum_n a_n \bar b_n t_n^{N+1} \nu_n  \ne 0.
\end{aligned}
\end{equation}
Then ${\rm dim}\, \big(\EE(\LL)\big)^\perp \le N$ and  
${\rm dim}\, \big(\EE(\LL^*)\big)^\perp \le N$.

3. If under conditions \eqref{mom1} we have 
$b\notin ' zL^2(\nu)$, then $\LL$ is complete. 
If $a\notin ' zL^2(\nu)$, then $\LL^*$ is complete.
\end{theorem}

Statement 1 of Theorem \ref{bio1} is a variation on the Keldysh--Matsaev theorems 
which deal with the so-called weak perturbations of the form $\A(I+S)$ or $(I+S)\A$ 
where $S$ is a compact operator of some class (note, however, that 
perturbations satisfying \eqref{mom0} need not be weak). For related results see
\cite[Theorem 1.1, Proposition 3.1]{by}. A result similar to Statement 2 
was proved in \cite[Theorem 1.1]{by1}. However, in all these results 
a restriction on geometry of the spectrum $\{s_n\}$ was imposed -- it was either real 
or contained in a finite union of rays. 
Here we get rid of any geometrical restrictions
by using new estimates of Cauchy transforms
of planar measures from \cite{bbb-fock} (see Subsection \ref{prel}). 

It was shown in \cite[Theorem 1.3]{by1} that for Hadamard-lacunary spectra   
one can prove a result converse to Statement 2 in Theorem \ref{bio1}:
if the operator is incomplete with infinite defect,  then
all moments are zero. This statement is no longer true 
when the lacunarity condition is relaxed (see \cite[Theorem 1.4]{by1}).

The second result gives conditions sufficient for 
completeness of $\LL^*$ when completeness of $\LL$ is known. 
Here the crucial assumption is the domination of the vector $b$ by $a$.

\begin{theorem}
\label{bio2}
Let $\A$ be a compact normal operator \textup(not necessarily in   
$\mathfrak{S}_p$ for some $p>0$\textup).
Assume that $\LL = \A + a\otimes b$ 
is complete and there exists $N\in \mathbb{N}_0$ such that 
\begin{equation}
\label{dom}
\sum_n \frac{|b_n|^2}{|a_n|^2 |t_n|^{2N}} <\infty.
\end{equation}
Then ${\rm dim}\, (\EE(\LL^*))^\perp \le N$.

If, moreover, $a\notin z L^2(\nu)$, then $\LL^*$ is complete.
\end{theorem}
\medskip


\subsection{Spectral synthesis.}
Here we state our positive results about the possibility of the spectral synthesis. 
As in the previous section, nonvanishing moments and a domination condition will 
play the role. The first result is similar to some 
results of Markus \cite[\S 2]{markus}, but again 
we do not need any restrictions on the spectrum location.

\begin{theorem}
\label{syn1}
Let $\A$ be a normal operator in the class $\mathfrak{S}_p$ 
and let $\LL = \A +a\otimes b$ be its rank one perturbation.
Assume that either $a_n \ne 0$ for any $n$ 
and $a\in z L^2(\nu)$, or $b_n \ne 0$ for any $n$ and $b\in z L^2(\nu)$.
If \eqref{mom0} is satisfied, then $\LL$ admits spectral synthesis.
\end{theorem}

\begin{theorem}
\label{syn2}
Let $\A$ be a normal operator in the class $\mathfrak{S}_p$ 
and let $\LL = \A +a\otimes b$ be its rank one perturbation.
Let either $a_n \ne 0$ for any $n$ or $b_n \ne 0$ for any $n$. 
Assume that there exists $N\in \mathbb{N}_0$ such that 
$\sum_n |a_n b_n| |t_n|^{N+1} \nu_n<\infty$ 
and either $N=0$ and condition \eqref{mom0} is satisfied 
or $N\ge 1 $ and conditions \eqref{mom1} are satisfied. Then 
for any $\LL$-invariant subspace $\MM$ we have
$$
{\rm dim}\, \big(\mathcal{M} \ominus \EE(\MM, \LL) \big) \le (N+1)^2.
$$
\end{theorem}

To state a ``domination'' result we need the following definition. We say that 
the sequence $T = \{t_n\}\subset \co$ is {\it power separated}
if there exists $M>0$ 
such that
\begin{equation}
\label{pow}
{\rm dist}\, (t_n, T\setminus \{t_n\}) \gtrsim |t_n|^{-M}.
\end{equation}
Condition \eqref{pow} implies that $T$ has a finite convergence exponent.

\begin{theorem}
\label{syn3}
Let $\LL = \A +a\otimes b$ be complete and 
let $T=\{t_n\}$ be power separated with exponent $M$. Assume that 
$a$ and $b$ satisfy \eqref{dom} and \textup(for the same $N$\textup)
\begin{equation}
\label{dom22}
|a_n|^2\nu_n \gtrsim |t_n|^{-2N-2}.
\end{equation}
Then for any $\LL$-invariant subspace $\MM$ we have
$$
{\rm dim}\, \big(\mathcal{M} \ominus \EE(\MM, \LL) \big) \le (M+N+1)^2.
$$
\end{theorem}
\smallskip
\noindent
{\bf Remark.} Conditions \eqref{dom} and \eqref{dom22} in Theorem \ref{syn3} 
can be replaced by a slightly stronger assumption $|a_n|^2 \nu_n \gtrsim |t_n|^{-2N}$,
which obviously implies \eqref{dom}. 
\medskip 

\subsection{Counterexamples.} 
We now turn to negative results. One of the main points of 
the present paper (as well as of \cite{by}) 
is that already for such small class as rank one perturbations
of normal operators one has a very rich and complicated 
spectral structure. We construct counterexamples showing 
that completeness of the adjoint operator or spectral 
synthesis may fail with any finite or infinite defect. 

\begin{theorem}
\label{coun}
Let $\A$ be a compact normal operator with
simple point spectrum and trivial kernel and let 
$N\in \mathbb{N} \cup \{\infty\}$. Then the following statements hold true.
\medskip

1. There exists a rank one perturbation $\LL$ of $\A$
such that $\ker \LL=\ker \LL^*=0$ and $\LL$ is complete, but $\LL^*$ 
is not complete and 
$$
{\rm dim}\, (\EE(\LL^*))^\perp = N.
$$

2. There exists a rank one perturbation $\LL$ of $\A$
such that 
\begin{enumerate} 
\item [(i)]
$\ker \LL=\ker \LL^*=0$\textup;
\item [(ii)] both $\LL$ and $\LL^*$ are complete\textup;
\item [(iii)] $\LL$ does not admit spectral synthesis
and, moreover, there exists $\LL$-invariant subspace $\MM$
such that
$$
{\rm dim}\, \mathcal{M} \ominus \EE(\MM, \LL) = N.
$$
\end{enumerate} 
\end{theorem}

Note that, for a bounded operator
$\mathcal{B}$,  if $\ker \mathcal{B} \ne 0$, but $\ker \mathcal{B}^* = 0$, then 
$\mathcal{B}^*$ is not complete. An explicit example of a complete compact operator 
$\mathcal{B}$ such that $\ker \mathcal{B} = \ker \mathcal{B}^*=0$, while $\mathcal{B}^*$
is not complete, was given by Deckard, Foia\c{s} and Pearcy  \cite{DeckFoPea}.
However, in their examples one cannot conclude that
the corresponding operator is a finite rank perturbation of a normal
operator. Surprisingly, one can find such examples
among rank one perturbations of
a compact normal operator with an arbitrary spectrum.

A concrete example of a rank one perturbation $\LL$ of a compact normal operator 
such that $\ker \LL = \ker \LL^*= 0$, $\LL$ is complete, but
$\LL^*$ is not, can be extracted from the results 
by A.A.~Lunyov and M.M.~Malamud \cite[Section 4]{LunMal15}. 
The operator in this example is realized as 
the inverse to a two-dimensional 
first order differential operator with specially chosen boundary conditions. 
A version of this example was presented in \cite[Appendix 1]{by}.

At the same time, Lunyov and Malamud \cite{LunMal14} showed that for a class of 
dissipative realizations of Dirac-type differential operators, completeness property 
is equivalent to the spectral synthesis property.
\medskip


\subsection{Ordered structure of invariant subspaces.}
Assume that both a compact operator $\LL$ and its adjoint $\LL^*$ are complete, 
but the spectral synthesis fails. If we denote by $\{x_n\}_{n\in N}$ the 
(generalized) eigenvectors 
of $\LL$ and by $\{y_n\}$ the eigenvectors of $\LL^*$ this means that
there exist an $\LL$-invariant subspace $\MM$ and $N_1\subset N$ such that
$x_n \in \MM$ if and only if $n\in N_1$, but $\ospan \{x_n: \, n\in N_1\} \ne \MM$. 
It is not difficult to show that 
$\MM^\perp \supset \{y_n:\, n\in N\setminus N_1 \}$ 
(see, e.g., the proof of Lemma 4.2 in \cite{markus}). Hence, 
\begin{equation}
\label{up}
\ospan \{x_n: \, n\in N_1\} \subset \MM \subset \{y_n:\, n\in N_2\}^\perp, 
\end{equation}
where $N_2 = N\setminus N_1$.
We say that all invariant subspaces $\MM$ satisfying \eqref{up} 
for a fixed set $N_1$
have {\it common spectral part} $\ospan \{x_n: \, n\in N_1\}$.

A natural problem is to describe all non-spectral invariant subspaces
or, at least, to find some structural properties of them. Apparently, 
in general, one cannot expect any structure of the lattice. However, 
for the case of rank one perturbations of normal operators,
there are good reasons to believe that the set of invariant subspaces $\MM$
satisfying \eqref{up} (i.e., the subspaces with the common
spectral part) is totally ordered by inclusion, that is, for any $\MM_1, \MM_2$
satisfying \eqref{up} one has $\MM_1\subset \MM_2$ or $\MM_2\subset \MM_1$. 
We state this as a conjecture.
\medskip
\\
{\bf Conjecture.} {\it  Let $\LL$ be a rank one perturbation of 
a compact normal operator. Then the set of all invariant subspaces 
with fixed common spectral part is totally ordered by inclusion.}
\medskip

We prove this conjecture in two cases: 
for rank one perturbations of compact {\it selfadjoint} operators 
(without any additional restrictions on the spectrum 
such as membership in a Schatten class)
and for the case 
of Schatten-class normal operators under certain conditions on 
the location of the spectrum. 

\begin{theorem}
\label{ord1}
Let $\A$ be a compact selfadjoint operator with
simple point spectrum and trivial kernel and let 
$\LL = \A +a\otimes b$ be its rank one perturbation such that
$b_n \ne 0$ for any $n$. Assume that $\LL$ and $\LL^*$ are complete. 
Then the set of all invariant subspaces 
with fixed common spectral part is totally ordered by inclusion.
\end{theorem}
\medskip 

In the case of normal operators we establish 
the ordered structure of invariant subspaces 
with fixed common spectral part when one of the following conditions holds:
\medskip
\begin{enumerate} 
\item [(i)] $\mathbf{Z:}$ $T$ is the zero set of some entire function 
of zero exponential type; 
\smallskip
\item [(ii)]
$\mathbf{\Pi:}$ $T$ lies in some strip and has finite convergence exponent;
\smallskip
\item [(iii)] $\mathbf{A_\gamma}:$  $T$ lies in some angle of size $\pi\gamma$, $0<\gamma<1$,
and the convergence exponent of $T$ is less than $\gamma^{-1}$.
\end{enumerate}

\begin{theorem}
\label{ord2}
Let $\A$ be a compact normal operator with
simple point spectrum $\{s_n\}$ 
such that its inverse spectrum $T=\{t_n\}$
satisfies one of the conditions 
$\mathbf{Z}$, $\mathbf{\Pi}$ or $\mathbf{A_\gamma}$. 
Let $\LL = \A +a\otimes b$ be its rank one perturbation such that
$b_n \ne 0$ for any $n$ and assume that $\LL$ and $\LL^*$ are complete. 
Then the set of all invariant subspaces 
with fixed common spectral part is totally ordered by inclusion.
\end{theorem}

In \cite{abb} the theory of Cauchy--de Branges spaces 
$\mathcal{H}(T, A, \mu)$ (see the definition in Subsection \ref{funk})
was developed which generalizes the theory of classical de Branges spaces. 
In particular, in the cases
$\mathbf{Z}$, $\mathbf{\Pi}$ and $\mathbf{A_\gamma}$ 
an ordering theorem for {\it nearly invariant subspaces} of the backward shift
(see the definition in Section \ref{th27})  in $\mathcal{H}(T, A, \mu)$ 
was obtained similar to the classical de Branges' Ordering Theorem \cite[Theorem 35]{br}. 
On the other hand, it is shown in \cite{abb} that in general 
there is no ordered structure for nearly invariant subspaces
in Cauchy--de Branges spaces.

As we will see, the functional model translates the ordering problem 
for invariant subspaces of rank one perturbations into
the ordering problem for nearly invariant subspaces 
in the spaces $\mathcal{H}(T, A, \mu)$.  
The proofs of Theorems \ref{ord1} and \ref{ord2} 
are variations on the beautiful idea 
used by L. de Branges in the proof of \cite[Theorem 35]{br}.        
\medskip


\subsection{Functional model.}         
\label{funk}
The model for rank one perturbations constructed in \cite{by} 
acted in some de Branges space. De Branges spaces' theory is a deep and important field 
which has numerous applications in operator theory and in 
spectral theory of differential 
operators. For the basics of de Branges theory we refer to 
L. de Branges' classical monograph \cite{br} and to \cite{rom}; 
some further results and applications can be found in \cite{mp, rem}. 
In the normal case their role is played by a more general class of 
{\it spaces of Cauchy transforms} that we will call {\it Cauchy--de Branges spaces}.

Let $T = \{t_n\}_{n=1}^\infty \subset \mathbb{C}$, where
$t_n$ are distinct, let 
$|t_n|\to \infty$ as $n\to\infty$, and let 
$\mu=\sum_n\mu_n\delta_{t_n}$ be a positive measure such that
$\sum_n \frac{\mu_n}{|t_n|^2 +1}<\infty$.
Also let $A$ be an entire function which has only simple zeros and whose zero set 
$\ZZ_A$ coincides with $T$. With any such $T$, $A$ and $\mu$
we associate the Cauchy--de Branges space $\mathcal{H}(T,A,\mu)$ of entire functions,
$$
\mathcal{H}(T,A,\mu):=\biggl{\{}f:f(z)= A(z)\sum_n\frac{a_n\mu^{1/2}_n}{z-t_n},
\quad a = \{a_n\}\in\ell^2\biggr{\}}
$$
equipped with the norm $\|f\|_{\mathcal{H}(T,A,\mu)}:=\|a\|_{\ell^2}$.
Note that the series in the definition 
of $\mathcal{H}(T,A,\mu)$ converge absolutely and uniformly on compact sets.

The spaces $\mathcal{H}(T,A,\mu)$ were introduced in full generality 
by Yu. Belov, T. Mengestie,
and K. Seip \cite{bms}. Essentially, they are spaces of Cauchy transforms. 
We need the function $A$ to get rid of poles and make the elements entire, but 
the spaces with the same $T$, $\mu$ and different $A$'s are isomorphic.
In what follows we will usually (but not always) assume that $T$ has 
a finite convergence exponent
and $A$ in this case will be chosen to be some canonical product of the corresponding order. 
We call the pair $(T, \mu)$ the {\it spectral data} for $\mathcal{H}(T,A,\mu)$. 

Each space $\mathcal{H}(T,A,\mu)$ is a reproducing kernel Hilbert space.
It is noted in \cite{bms} that if $\mathcal{H}$ is 
a reproducing kernel Hilbert space of entire functions such that
$\mathcal{H}$ has the {\it division property}  (that is, $\frac{f(z)}{z-w} 
\in \mathcal{H}$ whenever $f\in\mathcal{H}$ and $f(w) = 0$)
and there exists a Riesz basis of reproducing kernels in $\mathcal{H}$,
then $\mathcal{H} = \mathcal{H}(T,A,\mu)$ (as sets with equivalence of norms) 
for some choice of the parameters. Note that the functions 
$\overline{A'(t_n)} \mu_n \cdot \frac{A(z)}{z-t_n}$
form an orthogonal basis in $\mathcal{H}(T,A,\mu)$ and are 
the reproducing kernels at the points $t_n$.  
Reproducing kernels at other points can be written 
in a standard way using this orthogonal basis, but 
we do not have a good explicit formula for them. The reproducing kernel of the space 
$\mathcal{H}(T,A,\mu)$ at the point $\lambda$ will be denoted by $k_\lambda$.

In the case when $T\subset \RR$ and $A$ is real on $\RR$, 
the space $\mathcal{H}(T,A,\mu)$ is a de Branges space. This follows,
e.g., from the axiomatic description of de Branges spaces \cite[Theorem 23]{br}.
In a recent preprint \cite{abb} certain properties
of de Branges spaces (e.g., ordered structure of subspaces) are extended
to a class of Cauchy--de Branges spaces.

Following de Branges, we say that an entire function $G$ is {\it associated} to the space 
$\mathcal{H}(T,A,\mu)$  and write $G\in {\rm Assoc}\,(T, A, \mu)$ 
if, for any $F\in \mathcal{H}(T,A,\mu)$ and $w\in\CC$, 
we have 
$$
\frac{F(w)G(z) - G(w)F(z)}{z-w} \in \mathcal{H}(T,A,\mu).
$$
Equivalently, this means that $G \in \mathcal{P}_1 \mathcal{H}(T,A,\mu)$. We write
$F \in \mathcal{P}_n \mathcal{H}(T,A,\mu)$ if $F(z) = \sum_{j=0}^n z^j F_j(z)$, 
$F_j \in \mathcal{H}(T,A,\mu)$. Finally, if $G$ has zeros, then the inclusion 
$G\in {\rm Assoc}\,(T, A, \mu)$ is equivalent to
$\frac{G(z)}{z-\lambda} \in \mathcal{H}(T,A,\mu)$ for some (any) 
$\lambda\in \mathcal{Z}_G$. Note that, in particular, 
$A\in {\rm Assoc}\,(T, A, \mu) \setminus \mathcal{H}(T,A,\mu)$.

Now we are able to formulate the functional model of rank one perturbations 
of normal operators. Here and in what follows (except a part of Section \ref{funct}) 
we assume that $\A$ is a compact normal operator in a Hilbert space $H$ 
with simple point spectrum $\{s_n\}$, $s_n \ne 0$. We identify $H$ with $L^2(\nu)$, 
where $\nu = \sum_n \nu_n \delta_{s_n}$, 
and $\A$ with multiplication by $z$ in $L^2(\nu)$. 
The elements of $L^2(\nu)$ are identified with sequences., i.e., 
for $a\in L^2(\nu)$ we write $a=(a_n)$, where $a_n = a(s_n)$.

\begin{theorem}
\label{main1}
Let $\A$ be multiplication by $z$ in $L^2(\nu)$, 
$\nu = \sum_n \nu_n \delta_{s_n}$. Put $t_n = s_n^{-1}$. Let 
$\LL = \A + a\otimes b$ be a rank one perturbation of $\A$ such that $b=\{b_n\}
\in L^2(\nu)$ is a cyclic vector for $\A$, i.e., $b_n \ne 0$ for any $n$. Then
there exist 
\begin{itemize}
\item
a positive measure $\mu=\sum_n\mu_n\delta_{t_n}$ such that
$\sum_n \frac{\mu_n}{|t_n|^2 +1}<\infty$\textup; 
\item
a space $\mathcal{H}(T,A,\mu)$\textup;
\item 
an entire function $G\in  {\rm Assoc}\,(T, A, \mu)$ with $G(0)=1$
\end{itemize}
such that $\LL$ is unitarily equivalent to the model operator 
$\mathcal{T}_G: \mathcal{H}(T,A,\mu) \to \mathcal{H}(T,A,\mu)$, 
$$
(\mathcal{T}_Gf)(z) = \frac{f(z) - f(0)G(z)}{z}, \qquad f \in \mathcal{H}(T,A,\mu).
$$

Conversely, for any space $\mathcal{H}(T,A,\mu)$ and the function 
$G \in  {\rm Assoc}\,(T, A, \mu)$ with $G(0) = 1$ 
the corresponding operator $\mathcal{T}_G$ is a model
of a rank one perturbation for some compact normal operator $\A$
with spectrum $\{s_n\}$, $s_n = t_n^{-1}$.
\end{theorem}

As we will see in Lemma \ref{bro}, the assumption $b_n \ne 0$ 
does not lead to the loss of generality
when we study completeness of eigenvectors of $\LL$ and $\LL^*$.

The model itself is by no means original. In the case of selfadjoint 
operators it was constructed in \cite{by}, but many similar
models for rank one perturbations (in slightly different situations or under 
some additional restrictions) were known previously.
We will mention the model of V. Kapustin \cite{kap} for rank one perturbations 
of unitary operators and the model of G. Gubreev and A. Tarasenko 
for selfadjoint operators \cite{gub}. The latter model is especially 
close to ours with the same operator 
$\mathcal{T}_G$ as the model operator (in de Branges space setting).

The main novelty of the present work is not in the model but in its applications:
combined with recent developments in the theory of reproducing kernel 
Hilbert spaces of entire functions from \cite{bb, bbb, bbb1, bbb-fock, by} 
it leads to a more or less complete understanding 
of completeness and spectral synthesis problems for rank one perturbations
of normal \medskip
operators.


\subsection{Organization of the paper.}
The paper is organized as follows. In Section \ref{funct} 
we construct the functional model for rank one perturbations.         
Completeness of $\LL$ and $\LL^*$ is studied in Section \ref{comp1},
while in Section \ref{synte} positive results on the spectral synthesis 
(Theorems \ref{syn1}--\ref{syn3}) are proved. In Section \ref{th26}
we prove Theorem \ref{coun}. The Ordering Theorem for
invariant subspaces of rank one perturbations with common
spectral part is established in Section \ref{th27}. Finally, 
in Section \ref{volt} we discuss the description of compact normal operators
which have a Volterra rank one perturbation.
\bigskip


\section{Functional model}
\label{funct}

In this section we prove Theorem \ref{main1}. In fact, we construct 
a similar model
for general normal (not necessarily compact) operators.

\subsection{General normal operators}
Let $\A$ be the operator of multiplication by $z$ in $L^2(\nu)$ where $\nu$ 
is a finite measure with compact closed support $K$ and assume that
$K$ has zero planar Lebesgue measure. Here we do not assume that 
$\nu$ is an atomic measure. 
By the classical Hartogs--Rosenthal theorem, we have the following uniqueness
property:
$$
u\in L^2(\nu) \quad \text{and} \quad \int \frac{u(\zeta)d\nu(\zeta)}{\zeta-z} =0 
\ \ \text{for all} \ \ z\in \mathbb{C}\setminus K \ \Longrightarrow  \ u=0.
$$
Note that Wolff's example shows that 
there exist atomic measures with closed support 
$K= \{|z|\le 1\}$ such that the Cauchy transform of some nonzero function from $L^2(\nu)$
is identically zero in $\mathbb{C}\setminus K$. 

Consider the space of all Cauchy transforms
$$
\mathcal{C}(K, \nu) = \bigg\{ f(z) = \int\frac{u(\zeta)d\nu(\zeta)}{\zeta-z}, 
\ u\in L^2(\nu)  \bigg\},
$$
considered as the space of analytic functions in $\mathbb{C} \setminus K$ 
and equipped with the norm $\|f\|_{\mathcal{C}(K, \nu)} = \|u\|_{L^2(\nu)}$.
Then $\mathcal{C}(K, \nu)$ is a Hilbert space.

Note that for any $f\in \mathcal{C}(K, \nu)$ the mapping $f\mapsto (zf)_\infty$,
where $(zf)_\infty = \lim_{|z|\to\infty} zf(z)  =- \int ud\nu$ is a bounded linear 
functional on $\mathcal{C}(K, \nu)$.

\begin{theorem}
\label{main2}
Let $\A$ be multiplication by $z$ in $L^2(\nu)$, 
and $\LL = \A + a\otimes b$ be a rank one perturbation of $\A$ such that
$a, b\in L^2(\nu)$ and $b\ne 0$ $\nu$-a.e. Put $\sigma = |b|^2\nu$. 
Then there exists a function $\beta$ analytic in $\mathbb{C}\setminus K$ 
with the properties
\smallskip
\begin{enumerate} 
\item [(i)]
$\beta \notin \mathcal{C}(K, \sigma)$\textup;
\smallskip
\item [(ii)]
$\frac{\beta(z) - \beta(\lambda)}{z-\lambda} \in \mathcal{C}(K, \sigma)$ 
for any $\lambda\in \mathbb{C}\setminus K$\textup;
\smallskip
\item [(iii)]
$\beta(\infty)=1$,
\end{enumerate}
\smallskip
such that $\LL$ is unitarily equivalent to the model operator 
$\mathcal{M}_\beta: \mathcal{C}(K, \sigma) \to \mathcal{C}(K, \sigma)$, 
$$
(\mathcal{M}_\beta f)(z) = zf(z) - (zf)_\infty \beta(z), 
\qquad f \in \mathcal{C}(K, \sigma).
$$

Conversely, for any space $\mathcal{C}(K, \sigma)$, where $\sigma$ is a finite
Borel measure, and for any function $\beta$ having the properties 
{\rm (i)--(iii)}, the corresponding 
operator $\mathcal{M}_\beta$ is a model of a rank one perturbation of some
normal cyclic operator $\A$.
\end{theorem}

\begin{proof}
By the resolvent identity, for $u\in L^2(\nu)$,
$$
(\A - zI)^{-1}u - (\LL - zI)^{-1} u  = \big((\LL - zI)^{-1} u, b\big) (\A- zI)^{-1} a
$$
and so
$$
\big( (\A - zI)^{-1}u, b \big) - \big( (\LL - zI)^{-1} u, b \big)  = 
\big((\LL - zI)^{-1} u, b\big) \big(\A - zI)^{-1} a, b\big).
$$
Thus, 
$$
\big((\LL - zI)^{-1} u, b\big) = (\beta(z))^{-1}\big( (\A - zI)^{-1} u, b\big),
$$
where 
\begin{equation}
\label{beta}
\beta(z) = 1+ \big( (\A - zI)^{-1} a, b\big) = 
1+ \int\frac{a(\zeta) \bar b(\zeta)}{\zeta-z}d\nu(\zeta).
\end{equation}
The mapping
$$
V:\, u\mapsto  \big( (\A - zI)^{-1} u, b\big)  = 
\int\frac{u(\zeta)b^{-1}(\zeta)}{\zeta-z}d\sigma(\zeta) 
$$
is a unitary map from $L^2(\nu)$ to $\mathcal{C}(K, \sigma)$.
Let $z,w\in \rho(\LL)\cap\rho(\A)$ (where $\rho(\mathcal{U})$ denotes the resolvent 
set for an operator $\mathcal{U}$) and let $\beta(w) \ne 0$.  Then
$$
\begin{aligned}
V(\LL - wI)^{-1}u(z)  & = \beta(z)  
\big( (\LL - zI)^{-1} (\LL - wI)^{-1} u, b\big)  \\
& = \beta(z)
\frac{\big((\LL - zI)^{-1} u, b\big) - 
\big((\LL - wI)^{-1} u, b\big)}{z-w} \\
&  = 
\frac{(Vu)(z) - \dfrac{\beta(z)}{\beta(w)} (Vu)(w)}{z-w} 
\end{aligned}
$$

Now let $\mathcal{M}_\beta$ be the model operator
on $\mathcal{C}(K, \sigma)$. Let us compute 
$(\mathcal{M}_\beta-wI)^{-1}$ assuming $w\in\rho(\mathcal{M}_\beta)$, $w\notin K$,  
$\beta(w) \ne 0$.
Assume that $(\mathcal{M}_\beta - wI) g = (z-w)g - c_g \beta = h$,
where $g, h \in \mathcal{C}(K, \sigma)$,
$c_g = (zg)_\infty$. Then $g = \frac{h+c_g \beta}{z-w}$. Since $g, h, \beta$
are analytic outside $K$ and $\beta(w) \ne 0$ we conclude that $c_g = -g(w)/\beta(w)$
and so
$$
(\mathcal{M}_\beta-wI)^{-1} g(z) = 
\frac{g(z) - \dfrac{\beta(z)}{\beta(w)} g(w)}{z-w}. 
$$
Thus, 
$$
\big( V(\LL - wI)^{-1}u\big)(z) = \big((\mathcal{M}_\beta-wI)^{-1} Vu\big)(z)
$$
for any $w\notin K$
$w\in\rho(\mathcal{M}_\beta)\cap \rho(\LL)$, $\beta(w) \ne 0$. 
Since there are infinitely many such $w$, we conclude that 
$V\LL = \mathcal{M}_\beta V$.

Clearly, $\beta(\infty) = 1$ and so $\beta\notin \mathcal{C}(K, \sigma)$.
Also, 
$$
\frac{\beta(z)-\beta(\lambda)}{z-\lambda} = 
\int\frac{a(\zeta) b^{-1}(\zeta)}{(\zeta-\lambda)(\zeta-z)}d\sigma(\zeta) 
\in \mathcal{C}(K, \sigma)
$$
for any $\lambda\in\mathbb{C} \setminus K$.
\medskip

Let us prove the converse. Assume that $\beta$ satisfy conditions (i)--(iii).
Then, for some fixed $\lambda\in\mathbb{C} \setminus K$, there exists
$u\in L^2(\sigma)$ such that
$$
\beta(z) = \beta(\lambda) + (z-\lambda) \int \frac{u(\zeta)}{\zeta-z}d\sigma(\zeta) 
=
\beta(\lambda) - \int u(\zeta)d\sigma(\zeta)  + 
\int \frac{(\zeta-\lambda) u(\zeta)}{\zeta-z}d\sigma(\zeta). 
$$
From (iii) it follows that $\beta(\lambda) - \int u(\zeta)d\sigma(\zeta) = 1$.
It remains to take any $b\ne 0$ $\sigma$-a.e. and to
put $\nu=|b|^{-2}\sigma$ and $a(\zeta) = (\zeta-\lambda)u(\zeta)b(\zeta)$. 
Then $\beta$ is of the 
form \eqref{beta} and so it appears in the model for a rank one perturbation
of multiplication by $z$ in $L^2(\nu)$.
\end{proof} 
\medskip


\subsection{Preliminaries on Cauchy transforms and spaces $\mathcal{H}(T,A,\mu)$.}
\label{prel}
To apply the functional model to the study of completeness 
one needs to have good growth estimates for the 
Cauchy transforms of (discrete) planar measures. 
Powerfull tools for this were developed in \cite{bbb-fock}.
The following results from \cite{bbb-fock}
will be extensively used in what follows.

We say that $\Omega\subset \CC$ is a {\it set of zero area density} if 
$$
\lim_{R\to\infty} \frac{m_2(\Omega \cap D(0, R))}{R^2} = 0,
$$
where $m_2$ denotes the area Lebesgue measure in $\CC$. Clearly, a union
of two sets of zero density has zero density, a fact that we will constantly use.

The first statement shows that the Cauchy transform 
of a finite measure $\nu$ behaves asymptotically as
$\nu(\mathbb{C}) z^{-1}$ when $|z|\to \infty$. This is trivial for measures 
with compact support. The same is true in general up to a set of zero density.

\begin{lemma}\textup(\cite[Proof of Lemma 4.3]{bbb-fock}\textup)
\label{verd}
Let $\nu$ be a finite complex Borel measure
in $\CC$. Then, for any $\varepsilon>0$, 
there exists a set $\Omega$ of zero area density such that
$$
\bigg|\int_\mathbb{C}\frac{d\nu(\xi)}{z-\xi} - \frac{\nu(\mathbb{C})}{z} \bigg|
< \frac{\varepsilon}{|z|}, \qquad  z\in\CC\setminus\Omega.
$$
\end{lemma}
\medskip

The following variant of this statement will be often useful: 
if $\sum_n |t_n^{-1} d_n| <\infty$, then
\begin{equation}
\label{meas}
\sum_n \frac{d_n}{z-t_n}  = o(1), \qquad |z|\to\infty, \  
z\in\CC\setminus\Omega,
\end{equation}
for some set $\Omega$ of zero area density. This follows from
Lemma \ref{verd} and the formula
$$
\sum_n \frac{d_n}{z-t_n} = -\sum_n \frac{d_n}{t_n} + \sum_n \frac{d_n}{t_n(z-t_n)}.
$$

Let us mention one simple situation where we can conclude that
the Cauchy transform has the natural asymptotics along some rays. 

\begin{lemma}
\label{ugol}
Let $\nu$ be a finite complex Borel measure such that 
for some $\theta_0 \in [0, 2\pi)$ and $\delta>0$ we have 
${\rm supp}\,\nu\cap \{re^{i\theta}:\ r\ge 0, \ |\theta-\theta_0|
\le \delta\} = \emptyset$.
Then
$$
\int_\mathbb{C}\frac{d\nu(\xi)}{z-\xi} =  \frac{\nu(\mathbb{C})}{z}  +
o\Big(\frac{1}{z} \Big), \qquad z=re^{i\theta_0}, \ r\to\infty. 
$$
\end{lemma}

\begin{proof}
Since ${\rm dist}(z, {\rm supp}\,\nu) \asymp |z|$, $z=re^{i\theta_0}$, we have
$$
\bigg|\int_\mathbb{C}\frac{d\nu(\xi)}{z-\xi} - 
\frac{\nu(\mathbb{C})}{z} \bigg| \le 
\frac{1}{|z|}\int_\mathbb{C}\frac{|\xi| d|\nu|(\xi)}{|z-\xi|} = 
o\Big(\frac{1}{z} \Big), \qquad z=re^{i\theta_0}, \ r\to\infty,
$$
by the Dominated Convergence theorem.
\end{proof}

We also will need the following extension of the Liouville theorem.
This result which is due to A. Borichev appeared in 
\cite[Lemma 4.2]{bbb-fock} (in a slightly more general form). 
We include a short proof to make the exposition self-contained.

\begin{theorem}
\label{dens}
If an entire function $f$ of finite order is bounded on 
$\CC\setminus \Omega$ for some set $\Omega$ of zero area density, 
then $f$ is a constant. 
\end{theorem}

\begin{proof}
We prove an equivalent statement: {\it If $f$ is an entire function of finite order 
and $|f(z)|\to 0$ as $|z|\to \infty$, $z \notin \Omega$, for some 
set $\Omega$ of zero area density, then $f\equiv 0$.}

Assume that $f$ is non-zero. Then $\phi(z) = \log |f(z)|$ 
is subharmonic. Since $f$ takes arbitrarily large values, we may assume 
without loss of generality that $\phi(0) =1$. At the same time $\phi <0$ 
on $\CC\setminus \Omega$ for some set $\Omega$ of zero density. 
Let  $W(R)$ be the connected component of the open set
$\{z:\  \phi(z)>0\}\cap D(0, R)$ which contains the point 0
and let $S(R) = \{|z|=R\} \cap \partial W(R)$. Denote by $\sigma(R)$ the total length
of the arcs in $S(R)$. 

Denote by $\omega_{W(R)}(0, E)$ the harmonic measure at $0$ of 
$E\subset \partial W(R)$.  
By the Ahlfors--Carleman estimate \cite[Ch. IV, Th. 6.2]{gama}, 
there exists $r_0>0$ such that for $R>r_0$, 
$$
\omega_{W(R)}(0, S(R)) \le C\exp\bigg(-\pi\int_{r_0}^R\frac{dr}{\sigma(r)}\bigg),
$$
where $C>0$ is some absolute constant.
Recall that $f$ is of finite order. So, for some $C_1, N>0$,  
$\phi(z) \le C_1 R^N$, $|z|=R$. 
Since $\phi =0$ on $\partial  W(R)\setminus \{|z|=R\}$ and $\phi(z) \lesssim R^N$, 
$|z|=R$, we conclude by the ``Two Constants Theorem'' that
$$
\phi(0)\le C_1 R^N \omega_{W(R)}(0, S(R)) 
\le CC_1 R^N \exp\bigg(-\pi\int_{r_0}^R\frac{dr}{\sigma(r)}\bigg).
$$

It remains to show that since the set $\{z: \ \phi(z)>0\}$ has zero density, 
we have $\sigma(r) = o(r)$ on ``most of the circles''. This can be formalized as follows. 
Given $\vep>0$, for any sufficiently large $k$ (say, $k\ge n_0$) there exists
$E_k \subset [2^k, 2^{k+1}]$ with $|E_k|> 2^{k-1}$ such that for $r\in E_k$
we have $\sigma(r) <\vep r$. Otherwise, 
$$
m_2 (W(R) \cap\{2^k<|z|< 2^{k+1}\}) = \int_{2^k}^{2^{k+1}} \sigma(r) dr \ge 
\vep 2^{2k-1}, 
$$
a contradiction to the fact that $\Omega$ has zero area density.

Now, for $2^{n+1}\le R\le 2^{n+2}$ we have for some constant $C_2>0$,
$$
\begin{aligned}
\phi(0) &\le C_2 R^N \exp\bigg(-\pi\sum_{k=n_0}^n
\int_{E_k}  \frac{dr}{\sigma(r)}\bigg) \\
& \le C_2 R^N \exp\bigg(-\pi\sum_{k=n_0}^n
\int_{E_k}  \frac{dr}{\vep r}\bigg) \\
&\le C_2 R^N \exp\Big(-\frac{\pi}{4\vep} (n-n_0)\Big) \\
& \le C_2 R^N  \exp\Big(-\frac{\pi}{4\vep} (\log R-n_0)\Big).
\end{aligned}
$$
If $\vep$ is sufficiently small and $R$ is sufficiently large, we conclude that
$\phi(0) <1$, a contradiction.
\end{proof}

In \cite{abb} the following properties of functions in the spaces 
$\mathcal{H}(T,A,\mu)$ were discussed.

\begin{lemma} \textup(\cite[Lemma 2.5]{abb}\textup)
\label{gr1}
Let $A$ be an entire function of order $\rho$ 
with the zero set $T$. Then for any $\vep>0$ there exists
a set $E\subset (0,\infty)$ of zero linear density \textup(i.e., 
$|E \cap (0, R)| = o(R)$, $R\to\infty$, where 
$|e|$ denotes one-dimensional Lebesgue measure of $e\subset \RR$\textup) such that
for any entire function $f \in \mathcal{H}(T,A,\mu)$,
\begin{equation}
\label{cart}
|f(z)| \lesssim |z|^{\rho+1+\vep} |A(z)|, \qquad |z|\notin E.
\end{equation}

In particular, if $A$ is of order $\rho$ and of type $\varkappa$, then
any element of $\mathcal{H}(T,A,\mu)$  is of order at most $\rho$ 
and of type at most $\varkappa$ with respect to this order.
\end{lemma}

\begin{lemma}
\label{gr2}
If $f\in \mathcal{H}(T,A,\mu)$, then 
$$
\|f\|_{\mathcal{H}(T,A,\mu)}^2 = \sum_n 
\frac{|f(t_n)|^2}{|A'(t_n)|^2 \mu_n} 
$$
and there exists a set $\Omega$ of zero area density 
such that 
$$
|f(z)| = o(|A(z)|), \qquad |z|\to \infty, \ \ z\in\CC\setminus\Omega.
$$
\end{lemma}

\begin{proof}
Note that for $f= A\sum_{n}\frac{c_n\mu_n^{1/2}}{z-t_n}
\in \mathcal{H}(T,A,\mu)$ we have $f(t_n) = A'(t_n)c_n \mu_n^{1/2}$
and, by definition, $\|f\|_{\mathcal{H}(T,A,\mu)} = \|\{c_n\}\|_{\ell^2}$.

To prove that $|f(z)| = o(|A(z)|)$ recall that $\sum_n (1+|t_n|)^{-2}\mu_n<\infty$
whence $\sum_n |c_n t_n^{-1}|\mu_n^{1/2} <\infty$ and the estimate 
follows from \eqref{meas}.
\end{proof}

Using the above estimates one can state various criteria 
for the inclusion of $f$ into $\mathcal{H}(T,A,\mu)$.

\begin{theorem} \textup(\cite[Theorem 2.6]{abb}\textup)
\label{inc}
Let $\mathcal{H}(T,A,\mu)$ be a Cauchy--de Branges space
and let $A$ be of finite order. Then an entire function $f$ 
is in $\mathcal{H}(T,A,\mu)$ if and only if the following three conditions hold:
\begin{enumerate} 
\item [(i)] 
$\sum_n\dfrac{|f(t_n)|^2}{|A'(t_n)|^2 \mu_n} <\infty$\textup;      
\smallskip
\item [(ii)]
there exists a set $E\subset (0,\infty)$ 
of zero linear density and $N>0$ such that $|f(z)| \le |z|^N |A(z)|$,
$|z| \notin E$\textup;
\smallskip
\item [(iii)] there exists a set $\Omega$ of positive area density such that 
$|f(z)| = o(|A(z)|)$, $|z|\to \infty$, $z\in \Omega$.
\end{enumerate}
\end{theorem}

In many cases one can relax the conditions (ii)--(iii) and require the estimates
on a smaller set.
\medskip


\subsection{Proof of Theorem \ref{main1}}
Let $\nu = \sum_n \nu_n \delta_{s_n}$, $s_n\in \co$, $s_n \ne 0$,
$s_n\to 0$, and $t_n = s_n^{-1}$. Let $a, b\in L^2(\nu)$, $b_n\ne 0$ for any $n$
and $\sigma = |b|^2 \nu$. By Theorem \ref{main2} the perturbation $\LL = \A+a\otimes b$
is unitary equivalent to the model operator $\mathcal{M}_\beta$ on
$\mathcal{C}(K, \sigma)$ where $\beta$ is given by \eqref{beta}.
Now we put
$$
\mu_n = |t_n|^2 \sigma_n = |t_n|^2 |b_n|^2 \nu_n.
$$
Fix an entire function $A$ with zero set $\{t_n\}$ and $A(0) = 1$, 
and consider the space 
$\mathcal{H}(T,A,\mu)$. 

Put $G(z) = A(z) \beta(z^{-1})$. Then we have
\begin{equation}
\label{rep1}
\begin{aligned}
G(z)  = A(z) \beta(z^{-1}) & = A(z)\bigg( 1+ 
z \sum_n \frac{a_n \bar b_n t_n \nu_n}{z- t_n} \bigg) \\
& =
A(z) \bigg( 1+ 
\sum_n a_n \bar b_n t^2_n \nu_n \Big(\frac{1}{z- t_n} +\frac{1}{t_n} \Big) \bigg).
\end{aligned}
\end{equation}
It is easy to see that $G\in {\rm Assoc}\,(T, A, \mu)$ and $G(0) =1$.
\medskip

We will show that the mapping $U: f\mapsto A(z)z^{-1}f(z^{-1})$
maps $\mathcal{C}(K, \sigma)$ unitarily onto $\mathcal{H}(T,A,\mu)$ 
and realizes a unitary equivalence between 
$\mathcal{M}_\beta$ and $\mathcal{T}_G$.

Let $f(z) = \sum_n \frac{u_n \sigma_n}{z-s_n}$, $u=(u_n) \in L^2(\sigma)$ 
and $g(z) = A(z)z^{-1}f(z^{-1})$. Since
$$
z^{-1}f(z^{-1}) = \sum_n \frac{u_n t_n \sigma_n}{t_n-z} = 
\sum_n \frac{u_n t_n|t_n|^{-1} \sigma^{1/2}_n \mu_n^{1/2}}{t_n-z},
$$
we conclude that $g\in \mathcal{H}(T,A,\mu)$  and
$$
\|g\|_{\mathcal{H}(T,A,\mu)}^2 = \sum_n  |u_n|^2 \sigma_n = 
\|f\|^2_{\mathcal{C}(K, \sigma)}.
$$
Note also that $(zf)_\infty = \sum_n u_n \sigma_n = (Uf)(0)$.
Then we have 
$$
\begin{aligned}
(U\mathcal{M}_\beta f)(z) & = A(z)\big( 
z^{-2} f(z^{-1}) - (zf)_\infty z^{-1}\beta(z^{-1})\big)  \\
& = \frac{(Uf)(z) - (Uf)(0)G(z)}{z} = (\mathcal{T}_G Uf)(z).
\end{aligned}
$$
Thus, $\LL$ is unitary equivalent to $\mathcal{T}_G$. 

Finally, it is easy to see that any function $G \in {\rm Assoc}\,(T, A, \mu)$
with $G(0) = 1$ admits representation \eqref{rep1} for some $a, b, \nu$ such that 
$a, b \in L^2(\nu)$ and $\mu_n = |b_n|^2 |t_n|^2\nu_n$.
Therefore, the function $\beta$ such that $G(z)  = A(z) \beta(z^{-1})$ 
satisfies conditions (i)--(iii) of Theorem \ref{main2}. Thus, any such function
$G$ appears in the model of some rank one perturbation of the compact normal operator
with the spectral measure $\nu$. 
\qed
\medskip 

In what follows we will often use the following simple observation. 

\begin{lemma}
\label{gin}
The function $G$ given by \eqref{rep1}
belongs to $\mathcal{H}(T,A,\mu)$ if and only if 
$a\in zL^2(\nu)$ \textup(i.e., $\sum_n |a_n|^2|t_n|^2 \nu_n <\infty$\textup)
and
\begin{equation}
\label{degen}
1+ \sum_n  a_n \bar b_n t_n  \nu_n =0.
\end{equation}
\end{lemma}

\begin{proof} 
Assume that $G\in \mathcal{H}(T,A,\mu)$. Note that $G(t_n) = A'(t_n) t_n^2 a_n \bar b_n \nu_n$.
Then, by Lemma \ref{gr2} and the fact that $\mu_n = |t_n|^2 |b_n|^2 \nu_n$,
$$
\sum_n\frac{|G(t_n)|^2}{|A'(t_n)|^2 \mu_n} = \sum_n |a_n|^2|t_n|^2 \nu_n <\infty.
$$ 
Now the series $\sum_n |a_n b_n t_n|\nu_n$ converges and we may write
\begin{equation}
\label{rep2}
G(z) = A(z)
\bigg(1+ \sum_n a_n \bar b_n t_n \nu_n   + 
\sum_n \frac{a_n \bar b_n t^2_n \nu_n}{z- t_n} \bigg).
\end{equation}
Inclusion $A(z) \sum_n \frac{a_n \bar b_n t^2_n \nu_n}{z- t_n} \in \mathcal{H}(T,A,\mu)$
also follows from the condition $\sum_n |a_n|^2|t_n|^2 \nu_n <\infty$. Since, 
$A \notin \mathcal{H}(T,A,\mu)$ we conclude that the coefficient given 
by the left-hand side of \eqref{degen} is zero. 

The converse statement follows immediately from \eqref{rep2}.
\end{proof} 

It is easy to describe the point spectrum and eigenfunctions of the model operator
$\mathcal{T}_G$.  

\begin{lemma}
\label{eig}
$\eta \ne 0$ is an eigenvalue for $\mathcal{T}_G$ if and only if $\lambda = \eta^{-1}$ 
is a zero of $G$. The corresponding eigenvector of $\mathcal{T}_G$ is given
by $\frac{G}{z-\lambda}$ while the reproducing kernel 
$k_\lambda$ of $\mathcal{H}(T,A,\mu)$ 
is the eigenvector of $\mathcal{T}_G^*$ corresponding to $\bar\eta$. 
\end{lemma}

\begin{proof} 
We have $\mathcal{T}_G f= \eta f$ if and only if $f = c\frac{G}{1-\eta z} \in 
\mathcal{H}(T,A,\mu)$. Since $f$
is entire, this is equivalent to $G(\lambda) = 0$, $\lambda = \eta^{-1}$.
Note that $0$ is an eigenvalue of $\mathcal{T}_G$ 
if and only if $G\in \mathcal{H}(T,A,\mu)$. 

Now, if $\lambda \ne 0$ and $G(\lambda) =0$, then we have for any 
$f\in \mathcal{H}(T,A,\mu)$,
$$
(f, \mathcal{T}_G^* k_\lambda) =  
(\mathcal{T}_G f, k_\lambda)  
= \frac{f(z) - f(0) G(z)}{z}\bigg|_{z=\lambda} = \frac{f(\lambda)}{\lambda} = 
\frac{1}{\lambda}(f, k_\lambda),
$$
whence $\mathcal{T}_G^* k_\lambda = \bar \eta k_\lambda$.
\end{proof} 

\begin{remark}
{\rm  If $G$ has a zero $\lambda$ of multiplicity $m>1$, then
the corresponding root subspace for $\mathcal{T}_G$ 
is spanned by $\frac{G}{(z-\lambda)^j}$, $j=1, \dots, m$.
Similarly, one can find root vectors for $\mathcal{T}_G^*$ which are essentially
the reproducing kernels for the derivatives (see \cite{by} for details). 
To avoid uninteresting technicalities we assume in what follows that all zeros 
of $G$} {\it are simple}.
\end{remark}

Thus, the properties of rank one perturbations of normal compact operators
may be translated via the functional model to the geometric properties 
of systems of reproducing kernels in $\mathcal{H}(T,A,\mu)$, e.g.:
\begin{itemize}
\item
$\LL^*$ is complete if and only if the system $\{k_\lambda\}_{\lambda \in \mathcal{Z}_G}$
is complete in $\mathcal{H}(T,A,\mu)$ (equivalently, $\mathcal{Z}_G$  
is a {\it uniqueness set} for $\mathcal{H}(T,A,\mu)$);
\item
$\LL$ is complete if and only if the biorthogonal system
$\big\{\frac{G}{z-\lambda} \big\}_{\lambda \in \mathcal{Z}_G}$ is complete; 
\item  
$\LL$ admits spectral synthesis if and only if 
the system $\{k_\lambda\}_{\lambda \in \mathcal{Z}_G}$ 
is {\it hereditarily complete}
(see the definition in Section \ref{synte}).
\end{itemize}

Uniqueness sets in de Branges spaces (and in a more general setting of model subspaces 
of the Hardy space) were studied in 
\cite{bar06, fric, mp}. Completeness of systems biorthogonal to
systems of reproducing kernels was considered in \cite{bb}, while
in \cite{bbb, bbb1} a more or less complete understanding 
of hereditary completeness in de Branges spaces was achieved.
\bigskip


\section{Completeness of $\LL$ and $\LL^*$}
\label{comp1}

In this section we prove Theorems \ref{bio1} and \ref{bio2}
on completeness of rank one perturbations and their adjoints. 
First we remark that the condition $b_n\ne 0$ (which is required in
the functional model) does not lead to a loss of generality. 

\begin{lemma}
\label{bro}
Let $\A$ be a compact normal operator with simple spectrum $\{s_n\}_{n\in N}$, i.e.,
multiplication by $z$ in $H = L^2(\nu)$, $\nu = \sum_n \nu_n\delta_{s_n}$. Let 
$N=N_1\cup N_2$, $N_1\cap N_2 =\emptyset$. Then we can write $\A = \A_1\oplus \A_2$,
where $\A_j$ is multiplication by $z$ in $H_j = L^2(\nu|_{\{s_n\}_{n\in N_j}})$. 
Now let $\LL =\A+a\otimes b$. Assume that with respect to decomposition 
$H = H_1\oplus H_2$ we have $a=a_1\oplus a_2$, $b=b_1 \oplus 0$. Consider the operator 
$\LL_1 = \A_1 + a_1 \otimes b_1$ on $H_1$. Then
\begin{enumerate}
\item[(i)] 
if $\LL_1$ is complete, then $\LL$ is complete\textup;
\item[(ii)]
if $\LL_1^*$ is complete, then $\LL^*$ is complete.
\end{enumerate}
\end{lemma}

\begin{proof} For $u=u_1\oplus u_2$ we have
$$
\LL u = \big(\A_1 u_1 + (u_1, b_1)a_1\big) \oplus
\big(\A_2 u_2 + (u_1, b_1)a_2\big).
$$
Denote by $(e_m)_{m\in N_1}$ the standard orthogonal basis of $H_2$, 
$(e_m)_n = 0$, $m\ne n$, and $(e_m)_m=1$. It is clear that 
$0\oplus e_m$ is an eigenvector
of $\LL$ corresponding to the eigenvalue $s_m$, $m\in N_2$. 
Denote by $(f_k)$ the eigenvectors of 
$\LL_1$ corresponding to the eigenvalues $\lambda_k$. For simplicity we assume that
all eigenvalues are simple and also that $\{\lambda_k\} 
\cap \{s_m\}_{m\in N_2} = \emptyset$. If $u$ is an eigenvector of $\LL$ and $u_1\ne 0$, 
then  $u_1= f_k$ for some $k$. In this case for $u_2$ we have the equation:
$$
\A_2 u_2 + (u_1, b_1)a_2 = \lambda_k u_2,
$$
whence $u_2 = - (u_1, b_1) (\A_2 - \lambda_k I)^{-1} a_2$. In the case when $\lambda_k
= s_m$ for some $m$ we have a root vector instead of an eigenvector. We omit the details.

Thus, the eigenvectors of $\LL$ are of the form $\{0 \oplus e_m\} 
\cup \{f_k \oplus 
g_k\}$ for some $g_k\in H_2$. It is now obvious that this system is complete in $H$
if and only if the system $\{f_k\}$ is complete in $H_1$.

The statement for the adjoint operator is based 
on similar straightforward computations. Indeed, 
$$
\LL^* u = \big(\A_1^* u_1 + ((u_1, a_1)+(u_2, a_2))b_1\big) \oplus \A_2^* u_2.
$$
Let $f^*_k$ be the eigenvectors of $\LL_1^*$. 
Clearly, $u_1\oplus 0$ is an eigenvector of $\LL^*$
if and only if $u_1 = f_k^*$ for some $k$. If $u_2 \ne 0$
and $\LL^* u =\lambda u$, then $\lambda = \bar s_m$, $u_2 = e_m$ for some $m$
and $u_1$ can be found from the equation 
$(\LL_1^*- \bar s_m I) u_1 = - (e_m, a_2)b_1$. If $s_m \ne \lambda_k$ for any $k$, 
then there exists a unique vector $h_m$ such that $h_m\oplus e_m$ is an eigenvector of 
$\LL^*$ (in the case $s_m = \lambda_k$ there is a root vector). Again, it is obvious
that if the system $\{f_k^*\}$ is complete in $H_1$, then 
the system $\{f_k^* \oplus 0\} \cup\{h_m \oplus e_m \}$  
is complete in $H$. Note that the converse is not so clear.
\end{proof}
\medskip


\subsection{Proof of Theorem \ref{bio1}}
By Lemma \ref{bro} we may assume without loss of generality 
that $b_n\ne 0$ for any $n$. Then we can construct the functional model for $\LL$
in $\mathcal{H}(T,A,\mu)$.
In view of the symmetry of conditions, it is sufficient to prove the theorem for 
the adjoint operator $\LL^*$. In this case the eigenvectors are given by 
reproducing kernels $k_\lambda$, $\lambda \in \mathcal{Z}_G$.

Assume first that \eqref{mom0} holds, that is, the first moment is nonzero. 
If the system $\{k_\lambda\}_{\lambda \in \mathcal{Z}_G}$ is incomplete, 
then there exists a nonzero function $f\in \mathcal{H}(T,A,\mu)$ which vanish on 
$\mathcal{Z}_G$ and so $f=GU$ for some entire function $U$. 
By Lemma \ref{gr2} we have $|G(z)U(z)| = o(|A(z)|)$, $|z|\to \infty$,
outside a zero density set $\Omega_1$. On the other hand,
$$
G(z) = A(z)\bigg( 1+ \sum_n a_n \bar b_n t_n \nu_n + 
\sum_n \frac{a_n\bar b_n t_n^2 \nu_n}{z-t_n} \bigg).
$$
The last sum in brackets can be estimated by \eqref{meas} and we conclude that 
$|G(z)| \gtrsim |A(z)|$, $z\in \CC\setminus \Omega_2$, for some $\Omega_2$ 
of zero density. Hence, $|U(z)| = o(1)$ when $|z|\to \infty$, 
$z\in \CC\setminus \Omega$, for some set $\Omega$ of zero density.

Recall that all elements of $\mathcal{H}(T,A,\mu)$ are of finite order
not exceeding the order of $A$. Hence, $U$ is of finite order. 
By Theorem \ref{dens} $U\equiv 0$.

Now we assume that conditions \eqref{mom1}  are satisfied, i.e., 
the moment with number $N+1$ is the first nonzero moment. Using the elementary formula
$$
\frac{1}{z-t_n} = \frac{1}{z} +\frac{t_n}{z^2}+\dots +\frac{t_n^{m-1}}{z^m}
 +  \frac{t_n^m}{z^m(z-t_n)}
$$
we get 
$$
G(z) = \frac{A(z)}{z^{N-1}} \sum_n \frac{a_n \bar b_n t_n^{N+1} \nu_n}{z-t_n},
$$
whence, by Lemma \ref{verd}, $|G(z)|\gtrsim |z|^{-N}|A(z)|$, $z\notin \Omega$,
for some $\Omega$ of zero density. If $f=GU\in \mathcal{H}(T,A,\mu)$ then,
arguing as above, we conclude that $|U(z)| = o(|z|^N)$
as $|z|\to \infty$ outside a set of zero density
and so $U$ is a polynomial of degree at most $N-1$. Thus, the orthogonal complement to
$\{k_\lambda\}_{\lambda \in \mathcal{Z}_G}$ is contained in $\mathcal{P}_{N-1} G$
and so ${\rm dim}\, (\EE(\LL^*))^\perp \le N$.

It remains to show that if $a\notin ' zL^2(\nu)$, then $\LL^*$ is complete.
Indeed, by Lemma \ref{gin}, $G\notin \mathcal{H}(T,A,\mu)$. 
Therefore $GU \notin \mathcal{H}(T,A,\mu)$ for any polynomial $U$. Thus,
the orthogonal complement to $\{k_\lambda\}_{\lambda \in \mathcal{Z}_G}$ is trivial.
\qed 
\medskip


\subsection{Parametrization of the orthogonal complement 
to a system biorthogonal to a system of reproducing kernels.}
\label{param}
Let $\{k_{\lambda}\}_{\lambda\in \Lambda}$ be a minimal system in 
$\mathcal{H}(T,A,\mu)$. We assume that
$\{\lambda_n\}\cap T =\emptyset$. Let 
$G$ be an entire function which vanishes on $\Lambda$ and such that 
$\frac{G}{z-\lambda} \in \mathcal{H}(T,A,\mu)$. Such function exists due to minimality
of the system $\{k_{\lambda}\}_{\lambda\in \Lambda}$; it is possible that $G\notin 
\mathcal{H}(T,A,\mu)$, but $G\in {\rm Assoc}\,(T,A,\mu)$. 
Then it is clear that 
the system $\big\{ \frac{G(z)}{G'(\lambda)(z-\lambda)}\big\}_{\lambda\in \Lambda}$ 
is biorthogonal to $\{k_{\lambda}\}_{\lambda\in \Lambda}$.
The following parametrization of the orthogonal complement to the
biorthogonal system was suggested in \cite{bb}.

Assume that $h(z) = A(z) \sum_n \frac{c_n\mu_n^{1/2}}{z-t_n} \in \mathcal{H}(T,A,\mu)$
is orthogonal to the system 
$\big\{ \frac{G(z)}{z-\lambda}\big\}_{\lambda\in \Lambda}$.
Note that for $g, h \in \mathcal{H}(T,A,\mu)$ one has
$(g,h) = \sum_n \frac{g(t_n)\overline{h(t_n)}}{|A'(t_n)|^2\mu_n}$ 
whence
$$
\sum_n \frac{G(t_n) \bar c_n}{A'(t_n)\mu_n^{1/2} (t_n - \lambda)} =0, 
\qquad \lambda\in\Lambda.
$$
Therefore, the entire function 
$A(z) \sum_n \frac{G(t_n) \bar c_n}{A'(t_n)\mu_n^{1/2} (z-t_n)}$
vanishes on $\Lambda$ and we can write
\begin{equation}
\label{para1}
A(z) \sum_n \frac{G(t_n) \bar c_n}{A'(t_n)\mu_n^{1/2} (z-t_n)} = G(z)S(z) 
\end{equation}
for some entire function $S$. Note that, conversely, for any entire function
$S$ which satisfies equation \eqref{para1} with some sequence $(c_n)\in \ell^2$,
the function $h(z) = A(z) \sum_n \frac{c_n\mu_n^{1/2}}{z-t_n}$
belongs to $\mathcal{H}(T,A,\mu)$ and
is orthogonal to the system $\big\{ \frac{G(z)}{z-\lambda}\big\}_{\lambda\in \Lambda}$. 
We denote the class of all functions $S$ of the form \eqref{para1} by $\mathcal{S}$.
Note that $\mathcal{S}$ is a linear space. 

Comparing the values at $t_n$ we see that $G(t_n)S(t_n) = G(t_n) \bar c_n \mu_n^{-1/2}$, 
whence $S(t_n) = \bar c_n \mu_n^{-1/2}$ and so $\sum_n |S(t_n)|^2 \mu_n<\infty$.
When equipped with the norm $\|S\|^2 = \sum_n |S(t_n)|^2 \mu_n$
the space $\mathcal{S}$ becomes a Hilbert space
and the mapping 
$$
S\mapsto  A(z) \sum_n  \frac{\overline{S(t_n)} \mu_n}{z-t_n}
$$
is a unitary map from $\mathcal{S}$ onto the orthogonal complement to 
$\big\{ \frac{G(z)}{z-\lambda}\big\}_{\lambda\in \Lambda}$ in  
$\mathcal{H}(T,A,\mu)$.

We will need the following result from \cite{bb}:

\begin{lemma} \cite[Lemma 2.3]{bb}
\label{raz}
If $S\in\mathcal{S}$, then $\frac{S(z)-S(w)}{z-w}\in\mathcal{S}$ for
any $w\in \mathbb{C}$. 
\end{lemma}

\begin{proof}
Let $S \in \mathcal{S}$ and let $\lambda_0\in \Lambda$. Then we have
$$
\frac{G(z)S(z)}{A(z)} = \sum_n \frac{G(t_n) \bar c_n}{A'(t_n)\mu_n^{1/2} (z-t_n)}, \qquad 
\frac{G(z)}{A(z)} = (z-\lambda_0)\sum_n \frac{G(t_n)}{A'(t_n)(t_n-\lambda_0)(z-t_n)}.
$$
From this it is easy to show that
$$
\frac{1}{z-w}\bigg(\frac{G(z)S(z)}{A(z)} - \frac{G(w)S(w)}{A(w)} \bigg)
\qquad \text{and} \qquad
\frac{S(w)}{z-w} \bigg(\frac{G(w)}{A(w)} - \frac{G(z)}{A(z)} \bigg)
$$
have required representations.
\end{proof}
\medskip


\subsection{Proof of Theorem \ref{bio2}}
\label{prbio1}
We will need the following lemma.
                                        
\begin{lemma}
\label{pol1}
Let $d_n$ be such that
$$
\sum_n \frac{|d_n|}{|t_n|} <\infty \qquad \text{and} 
\qquad \sum_n \frac{|d_n|^2}{|t_n|^{2N}\mu_n}<\infty
$$
for some $N\in \mathbb{N}_0$. Then 
$f(z) = A(z) \sum_n \frac{d_n}{z-t_n} 
\in \mathcal{P}_N \mathcal{H}(T,A,\mu)$.
\end{lemma}

\begin{proof}
Since 
\begin{equation}
\label{srt1}
\frac{1}{z-t_n} = - \frac{1}{t_n} -\frac{z}{t_n^2}-\dots -\frac{z^{N-1}}{t_n^N}
 +  \frac{z^N}{t_n^N (z-t_n)},
\end{equation}
we have
$$
f(z) = A(z)P(z) + z^N A(z)\sum_n \frac{d_n}{t_n^N (z-t_n)},
$$
where $P$ is a polynomial of degree at most $N-1$.
By the hypothesis $A(z)\sum_n \frac{d_n}{t_n^N (z-t_n)} \in \mathcal{H}(T,A,\mu)$
Also, $A\in  {\rm Assoc}\,(T, A, \mu)$ and so $AP\in \mathcal{P}_N \mathcal{H}(T,A,\mu)$.
\end{proof}
\medskip
\begin{proof}[Proof of Theorem \ref{bio2}]
By Lemma \ref{bro} we may assume that $a_n, b_n \ne 0$ for any $n$. 
Note that $\LL^* = \A^* + b\otimes a$ also is a rank one perturbation of
a normal operator. By the symmetry, we can prove the following 
statement which is equivalent to Theorem \ref{bio2}: 
\medskip

{\it If $\LL^*$ is complete and 
$$
\sum_n \frac{|a_n|^2}{|b_n|^2 |t_n|^{2N}} <\infty,
$$
then ${\rm dim}\, (\EE(\LL))^\perp \le N$. If, moreover,
$b\notin zL^2(\nu)$, then $\LL$ is complete.}
\medskip

Consider the functional model for $\LL$ 
in the Cauchy--de Branges space $\mathcal{H} = \mathcal{H}(T, A, \mu)$.
Recall that the eigenvectors of $\LL$ are of the form $\frac{G}{z-\lambda}$, 
$\lambda\in \mathcal{Z}_G$, where
$$
G(z) = A(z)\bigg(1+ z\sum_n \frac{a_n \bar b_n  t_n \nu_n}{z-t_n}\bigg),
$$
while the eigenfunctions of $\LL^*$ are the reproducing kernels $k_\lambda$, 
$\lambda\in \mathcal{Z}_G$.

By the discussion in the Subsection \ref{param}, if $f(z) = A(z)
\sum_n \frac{c_n \mu_n^{1/2}}{z- t_n}$ is orthogonal to 
$\big\{\frac{G}{z-\lambda}\big\}_{\lambda\in \mathcal{Z}_G}$, 
then there exists an entire function $S$ from the 
corresponding space $\mathcal{S}$ such that
$$
G(z)S(z) = A(z) \sum_n \frac{G(t_n) \bar c_n}{A'(t_n) 
\mu_n^{1/2}(z-t_n)} = 
A(z) \sum_n \frac{d_n}{z-t_n}.
$$
Note that the coefficients $d_n$ satisfy
\begin{equation}
\label{d1}
|d_n| = \bigg| 
\frac{G(t_n)c_n}{A'(t_n) \mu_n^{1/2}}\bigg| =  
\bigg|\frac{a_n b_n t_n^2 \nu_n c_n}{\mu_n^{1/2}}\bigg| = 
|a_n t_n \nu_n^{1/2}c_n|.
\end{equation}
Hence,                             
\begin{equation}
\label{d2}
\sum_n \frac{|d_n|^2}{|t_n|^{2N} \mu_n} \lesssim 
\sum_n \frac{|a_n|^2}{|b_n|^2 |t_n|^{2N}} <\infty,
\end{equation}
and $GS\in \mathcal{P}_N \mathcal{H}(T, A, \mu)$ 
by Lemma \ref{pol1}.

If $S$ has at least $N$ zeros $z_1, \dots z_N$ counting multiplicities,
then 
$$
(z-z_1)^{-1}\dots(z-z_N)^{-1} G(z)S(z) 
\in \mathcal{H}(T, A, \mu).
$$ 
This contradicts the fact that
$\LL$ is complete and so the set $\mathcal{Z}_G$ is a uniqueness set for 
$\mathcal{H}(T,A, \mu)$. 
Thus, $S = Q_1 e^{Q_2}$ where $Q_1$ is a polynomial of degree 
at most $N-1$ and $Q_2$ is  some entire function. 

By Lemma \ref{raz}, $\frac{S(z) - S(w)}{z-w} \in \mathcal{S}$ 
for any $S\in \mathcal{S}$.
If $S = Q_1 e^{Q_2}$ and $Q_2\ne const$, then  there exists $w\in \mathbb{C}$
such that $\frac{S(z) - S(w)}{z-w}$ have infinitely many zeros 
and, repeating the above argument, we again come to a contradiction. 
We conclude that 
$\mathcal{S} \subset \mathcal{P}_{N-1}$ and so
${\rm dim}\, (\EE(\LL))^\perp \le N$. 

Note that $S(t_n) = \bar c_n \mu_n^{-1/2}$ and so
$$
\sum_n |S(t_n)|^2 |b_n|^2 |t_n|^2 \nu_n <\infty.
$$
If $S$ is a polynomial and 
$b\notin zL^2(\nu)$, then $S\equiv 0$ and so $\LL$ is complete. 
\end{proof}

\begin{remark}
{\rm For the case of selfadjoint operators a result similar to Theorem \ref{bio2}
was proved in \cite{by} with pointwise (in place of integral) domination 
and for $T$ with finite convergence exponent. It is obvious that in this case
each of the conditions
$|a_n|^2 \nu_n \gtrsim |t_n|^{-N}$ or 
$|b_n| \lesssim |t_n|^N|a_n|$ of \cite{by} implies condition \eqref{dom}.  }
\end{remark}
\bigskip


\section{Spectral synthesis}
\label{synte}

\subsection{Hereditary complete systems.}
The spectral synthesis for an operator is equivalent to a certain
``strong completeness'' property of its root vectors.
Let $\{x_n\}_{n\in N}$ be a complete and minimal system 
in a separable Hilbert space $H$ 
and let $\{\tilde x_n\}_{n\in N}$ be its biorthogonal system. 
The system $\{x_n\}$ is said to be {\it hereditarily complete} 
(or to be a {\it strong Markushevich basis}, or to admit {\it spectral synthesis})  
if $x\in \ospan \{(x,\tilde x_n) x_n\}$ for any $x\in H$. In other words,
any $x$ can be approximated by partial sums of its Fourier series
with respect to the biorthogonal pair $\{x_n\}$, $\{\tilde x_n\}$.

An equivalent definition of a hereditarily complete system is that
for any partition $N = N_1 \cup N_2$, $N_1 \cap N_2 =\emptyset$, of the index set $N$,
the {\it mixed system}
$$
\{x_n\}_{n\in N_1} \cup \{\tilde x_n\}_{n\in N_2}
$$
is complete in $H$. 

By a theorem of A. Markus \cite[Theorem 4.1]{markus}, a compact operator with complete set 
of root vectors $\{x_n\}$ admit the spectral synthesis if and only if 
the system $\{x_n\}$ is hereditarily complete. For 
a survey  of hereditary completeness and its relations to operator theory 
we refer to \cite[Chapter 4]{hrnik} (see also \cite{bbb1} and references therein).
\medskip


\subsection{Hereditary completeness for systems of reproducing kernels.}
\label{her1}
Now let $\LL$ be a rank one perturbation of a compact normal operator and let 
$\mathcal{T}_G$ be its functional model in a space $\mathcal{H}(T,A,\mu)$. 
Now the possibility of spectral synthesis for $\LL$ reduces to the hereditary 
completeness of the system 
$\big\{ \frac{G}{z-\lambda}\big\}_{\lambda\in \mathcal{Z}_G}$
(equivalently, $\{k_{\lambda}\}_{\lambda\in \mathcal{Z}_G}$),
that is, completeness of all mixed systems. A method for the study 
of hereditary completeness of systems  of reproducing kernels in de Branges spaces
was developed and successfully applied in \cite{bbb, bbb1}. In particular,
in \cite{bbb} a long-standing problem of the spectral synthesis for exponential systems
was solved. In \cite{by} these results were used to study spectral synthesis for 
rank one perturbations of compact selfadjoint operators.

We will see that these methods apply to the Cauchy--de Branges spaces as well. Let
$\{k_{\lambda}\}_{\lambda\in \Lambda}$ be a complete and minimal system of reproducing 
kernels and let $\big\{ \frac{G(z)}{z-\lambda}\big\}_{\lambda\in \Lambda}$ 
be its biorthogonal system. Here $G$ is the unique (up to multiplication by a constant) 
function in ${\rm Assoc}\,(T, A, \mu)$ such that $\mathcal{Z}_{G} = \Lambda$.
For the partition $\Lambda = \Lambda_1\cup \Lambda_2$, 
$\Lambda_1\cap \Lambda_2 = \emptyset$, consider the corresponding mixed system
\begin{equation}
\label{mix}
\mathcal{K}(\Lambda_1, \Lambda_2) :=
\{k_{\lambda}\}_{\lambda\in \Lambda_1}
\cup 
\Big\{ \frac{G}{z-\lambda}\Big\}_{\lambda\in \Lambda_2}. 
\end{equation}
We assume that $\Lambda\cap T = \emptyset$. This is not a restriction since 
both properties of being hereditarily complete or to be a Riesz basis of reproducing 
kernels in $\mathcal{H}(T,A,\mu)$ are stable under small perturbations of points.

One can parametrize the orthogonal complement to the system \eqref{mix} similarly to 
Subsection \ref{param}. Choose two functions $G_1$ and $G_2$
such that $G=G_1G_2$, $\mathcal{Z}_{G_1} = \Lambda_1$, $\mathcal{Z}_{G_2} = \Lambda_2$.
Assume that $f(z) = A(z)\sum_n \frac{c_n\mu_n^{1/2}}{z-t_n} \in \mathcal{H}(T,A,\mu)$
is orthogonal to the system \eqref{mix}. 
The fact that 
$f\perp \{k_{\lambda}\}_{\lambda\in \Lambda_1}$ is equivalent to $f=G_1S_1$ 
for some entire function $S_1$. As in Subsection \ref{param}, 
the orthogonality $f\perp \big\{ \frac{G(z)}{z-\lambda}\big\}_{\lambda\in \Lambda_2}$
can be rewritten as 
$$
\sum_n \frac{G(t_n) \bar c_n}{A'(t_n)\mu_n^{1/2} (t_n - \lambda)} =0, 
\qquad \lambda\in\Lambda_2,
$$
and so the entire function $A(z) \sum_n \frac{G(t_n) 
\bar c_n}{A'(t_n)\mu_n^{1/2} (z-t_n)}$ is divisible by $G_2$, 

We conclude that
$f(z) = A(z)\sum_n \frac{c_n\mu_n^{1/2}}{z-t_n} \in \mathcal{H}(T,A,\mu)$
is orthogonal to the system \eqref{mix} if and only if there exist two entire functions
$S_1, S_2$ such that we have two interpolation formulas:
\begin{equation}
\label{para2}
\begin{aligned}
& A(z)\sum_n \frac{c_n\mu_n^{1/2}}{z-t_n}  = G_1(z)S_1(z), \\ 
& A(z) \sum_n \frac{G(t_n) \bar c_n}{A'(t_n)\mu_n^{1/2} (z-t_n)}  = G_2(z)S_2(z).
\end{aligned}
\end{equation}
Conversely, if there exist two entire functions $S_1, S_2$ satisfying \eqref{para2}
for some $(c_n) \in \ell^2$, then 
$f(z) = A(z)\sum_n \frac{c_n\mu_n^{1/2}}{z-t_n}$ is orthogonal to the system \eqref{mix}. 
We denote by $\mathcal{S}_{12}$ the set of all pairs $(S_1, S_2)$ satisfying \eqref{mix}, 
this set parametrizes the orthogonal complement to \eqref{mix}. Note that
if $(S_1, S_2)$ and $(\tilde S_1, \tilde S_2)$ are in $\mathcal{S}_{12}$, then 
$(S_1+ \tilde S_1, S_2 +  \tilde S_2)\in  \mathcal{S}_{12}$. However, $\mathcal{S}_{12}$
is not a linear space: for 
$(S_1, S_2) \in \mathcal{S}_{12}$ and $\alpha\in\CC$, 
we have $(\alpha S_1, \bar \alpha S_2) \in \mathcal{S}_{12}$.

Comparing the values at $t_n$, we get
$$
A'(t_n) c_n \mu_n^{1/2} = G_1(t_n)S_1(t_n), \qquad G(t_n)\bar c_n \mu_n^{-1/2} = 
G_2(t_n)S_2(t_n),
$$
whence
$$
S_1(t_n)S_2(t_n) = |c_n|^2 A'(t_n).
$$
Hence, if we put $S=S_1S_2$, we see that the entire functions $S$ 
and $A \sum_n \frac{|c_n|^2}{z-t_n}$ coincide on $T$. Thus, there exists
an entire function $R$ such that
$$
S(z) = A(z) \bigg(\sum_n \frac{|c_n|^2}{z-t_n} + R(z)\bigg).
$$
This representation will play the key role in 
the proofs of Theorems \ref{syn1}--\ref{syn3}. Note that in the case when 
$A$ is of finite order all functions in the space  
$\mathcal{H}(T, A, \mu)$  (for any admissible measure $\mu$) 
are of finite order by Lemma \ref{gr1}. 
In particular, $G$, $G_1S_1$, $G_2 S_2$ are of
finite order and so $R$ is of finite order.

The following observation also will be useful. Let 
$(\tilde S_1, \tilde S_2)$ be another element of $\mathcal{S}_{12}$ 
corresponding to a function 
$g(z) =A(z)\sum_n \frac{d_n\mu_n^{1/2}}{z-t_n}$ orthogonal to
$\mathcal{K}(\Lambda_1, \Lambda_2)$. Then, analogously,
$$
A'(t_n) d_n \mu_n^{1/2} = G_1(t_n)\tilde S_1(t_n), \qquad G(t_n)\bar d_n 
\mu_n^{-1/2} = G_2(t_n)\tilde S_2(t_n),
$$
and so $S_1(t_n)\tilde S_2(t_n) = c_n \bar d_n A'(t_n)$, 
$\tilde S_1(t_n) S_2(t_n) = \bar c_n d_n A'(t_n)$, whence
$$
\frac{S_1(z)\tilde S_2(z)}{A(z)} =  \sum_n \frac{c_n \bar d_n}{z-t_n} + U(z),
\qquad
\frac{\tilde S_1(z) S_2(z)}{A(z)} =  \sum_n \frac{\bar c_n d_n}{z-t_n} + V(z),
$$
for some entire functions $U$ and $V$.


\subsection{Proof of Theorem \ref{syn1}}
Assume that $b_n \ne 0$ for any $n$ and $b\in zL^2(\nu)$.
Then for $\mu$ defined by $\mu_n = |b_n|^2|t_n|^2 \nu_n$ we have
$\sum_n \mu_n <\infty$.

Let $\mathcal{T}_G$ be the model operator in $\mathcal{H}(T, A, \mu)$
unitarily equivalent to $\LL$. We need to show that $\mathcal{T}_G$  
admits the spectral synthesis, i.e., that any mixed system
$\mathcal{K}(\Lambda_1, \Lambda_2)$ 
(where $\Lambda_1\cup\Lambda_2 = \Lambda = \mathcal{Z}_G$)
is complete in $\mathcal{H}(T, A, \mu)$. 
Note that the systems 
$\{k_{\lambda}\}_{\lambda\in \Lambda}$ and 
$\big\{ \frac{G(z)}{z-\lambda}\Big\}_{\lambda\in \Lambda}$ 
are complete by Theorem \ref{bio1}.

As in Theorem \ref{bio1}, we have 
$$
G(z) = A(z)\bigg( 1+ \sum_n a_n \bar b_n t_n \nu_n + 
\sum_n \frac{a_n\bar b_n t_n^2 \nu_n}{z-t_n} \bigg),
$$
whence, by \eqref{meas}, 
$|G(z)| \asymp |A(z)|$, $z\in \CC\setminus \Omega_1$ for some $\Omega_1$ 
of zero density. Assume that 
$f(z) = A(z)\sum_n \frac{c_n\mu_n^{1/2}}{z-t_n} 
\in \mathcal{H}(T, A, \mu)$ is orthogonal to 
$\mathcal{K}(\Lambda_1, \Lambda_2)$
and let $(S_1, S_2) \in \mathcal{S}_{12}$ be the corresponding entire functions 
for which \eqref{para2} holds. Multiplying the equations in  \eqref{para2} we get
\begin{equation}
\label{mup}
\frac{G(z)}{A(z)} \bigg( \sum_n \frac{|c_n|^2}{z-t_n} + R(z) \bigg) = 
\bigg(\sum_n \frac{c_n\mu_n^{1/2}}{z-t_n}\bigg) \cdot
\bigg(\sum_n \frac{G(t_n) \bar c_n}{A'(t_n)\mu_n^{1/2} (z-t_n)}\bigg), 
\end{equation}
which can be rewritten as 
\begin{equation}
\label{mup1}
\frac{G(z)}{A(z)}R(z)  =  
\bigg(\sum_n \frac{c_n\mu_n^{1/2}}{z-t_n}\bigg) \cdot
\bigg(\sum_n \frac{G(t_n) \bar c_n}{A'(t_n)\mu_n^{1/2} (z-t_n)}\bigg)
- \frac{G(z)}{A(z)}  \sum_n \frac{|c_n|^2}{z-t_n}.
\end{equation}
By \eqref{meas}, the right-hand side in \eqref{mup1}
is $o(1)$ as $|z|\to\infty$ 
outside a set of zero density. 
Hence, $|R(z)| = o(1)$, $|z|\to \infty$, $z\notin \Omega_2$, $\Omega_2$ 
of zero density. Since $R$ is of finite order (see Subsection \ref{her1}), 
$R\equiv 0$ by Theorem \ref{dens}.  Thus, we have
\begin{equation}
\label{srt}
\frac{G(z)}{A(z)} \sum_n \frac{|c_n|^2}{z-t_n}  = 
\bigg(\sum_n \frac{c_n\mu_n^{1/2}}{z-t_n}\bigg) \cdot
\bigg(\sum_n \frac{G(t_n) \bar c_n}{A'(t_n)\mu_n^{1/2} (z-t_n)}\bigg).
\end{equation}
By Lemma \ref{verd} and the fact that
$|G(z)| \asymp |A(z)|$ outside a set of zero density, the modulus of the 
left-hand side is $\gtrsim |z|^{-1}$ outside a set
of zero density. The second term in the right-hand side of \eqref{srt}
is $o(1)$ as $|z|\to\infty$ 
outside a set of zero density. Finally, note that $(c_n \mu_n^{1/2}) \in \ell^1$
since $\sum_n \mu_n<\infty$. Hence, by Lemma \ref{verd},
$$
\sum_n \bigg|\frac{c_n\mu_n^{1/2}}{z-t_n}\bigg| \lesssim \frac{1}{|z|} 
$$
outside a set of zero density. Thus, the right-hand side 
of \eqref{srt} is $o(|z|^{-1})$ as $|z|\to\infty$ 
outside a set of zero density, a contradiction.
\qed
\medskip


\subsection{Proof of Theorem \ref{syn2}}

We will prove the equivalent result for the model operator
$\mathcal{T}_G$ in $\mathcal{H}(T, A, \mu)$. 
Let $\MM$ be $\mathcal{T}_G$ invariant and let 
$\EE(\MM, \mathcal{T}_G) = \ospan \big\{\frac{G}{z-\lambda}: 
\ \lambda\in \Lambda_2 \big\}$, where $\Lambda_2\subset \Lambda = \mathcal{Z}_G$. 
Then it is not difficult to show that  
$\MM^\perp \supset \{k_\lambda:\ \lambda\in \Lambda_1 = 
\Lambda\setminus\Lambda_2\}$ (see, e.g., the proof of Lemma 4.2 in 
\cite{markus}). Hence, 
$$
\ospan \Big\{\frac{G}{z-\lambda}: 
\ \lambda\in \Lambda_2 \Big\} \subset \MM \subset 
\big( \ospan \{k_\lambda:\ \lambda\in \Lambda_1\}\big)^\perp.
$$ 
Thus, ${\rm dim}\, \mathcal{M} \ominus \EE(\MM, \mathcal{T}_G)$
does not exceed the dimension of the orthogonal complement to the mixed system
$\mathcal{K}(\Lambda_1, \Lambda_2) =
\{k_{\lambda}\}_{\lambda\in \Lambda_1} \cup 
\big\{ \frac{G(z)}{z-\lambda}\Big\}_{\lambda\in \Lambda_2}$.
It remains to show that this dimension admits an estimate
depending on $N$ only. 
\medskip
\\
{\bf Step 1.} Let $f(z) = A(z)\sum_n \frac{c_n\mu_n^{1/2}}{z-t_n} 
\in \mathcal{H}(T, A, \mu)$ be orthogonal to $\mathcal{K}(\Lambda_1, \Lambda_2)$
and let $(S_1, S_2) \in \mathcal{S}_{12}$ be the corresponding entire functions 
for which \eqref{para2} holds. As in the proof of Theorem \ref{syn1},
multiplying the equations in  \eqref{para2}
we get \eqref{mup}--\eqref{mup1}.
By \eqref{meas}, the right-hand side in \eqref{mup1}
is $o(1)$ as $|z|\to\infty$, 
$z\notin\Omega_1$ for some set $\Omega_1$ of zero density. 
We use the fact that $|G|\lesssim |A|$ outside a set of zero density. 

We consider the case $N\ge 1$, the case $N= 0 $ is analogous.
As in the proof of Theorem \ref{bio1}, 
$$
G(z) = \frac{A(z)}{z^{N-1}} \sum_n \frac{a_n \bar b_n t_n^{N+1} \nu_n}{z-t_n},
$$
whence, by Lemma \ref{verd}, $|G(z)|\gtrsim |z|^{-N}|A(z)|$, 
$z\notin \Omega_2$,
for some $\Omega_2$ of zero density. Thus, $|R(z)| = o(|z|^N)$
as $|z|\to\infty$, $z\notin\Omega$, $\Omega$ of zero density. 
Recall that $R$ is of finite order.  Then, 
by Theorem \ref{dens}, $R$ is a polynomial of degree at most $N-1$.
\medskip
\\
{\bf Step 2.} Let $f(z) = A(z)\sum_n \frac{c_n\mu_n^{1/2}}{z-t_n}$ and $g(z)
=A(z)\sum_n \frac{d_n\mu_n^{1/2}}{z-t_n}$ be two mutually orthogonal functions 
in $(\mathcal{K}(\Lambda_1, \Lambda_2))^\perp$ and let corresponding
elements $(S_1, S_2)$ and $(\tilde S_1, \tilde S_2)$  of $\mathcal{S}_{12}$ 
satisfy
$$
\frac{S_1(z)S_2(z)}{A(z)} =  \sum_n \frac{|c_n|^2}{z-t_n} + R(z), \qquad 
\frac{\tilde S_1(z)\tilde S_2(z)}{A(z)} = 
 \sum_n \frac{|d_n|^2}{z-t_n} + \tilde R(z),
$$
$$
\frac{S_1(z)\tilde S_2(z)}{A(z)} =  \sum_n \frac{c_n \bar d_n}{z-t_n},
\qquad
\frac{\tilde S_1(z) S_2(z)}{A(z)} = \sum_n \frac{\bar c_n d_n}{z-t_n},
$$
i.e. in the ``cross-products'' $S_1 \tilde S_2$ 
and $\tilde S_1 S_2$ there are no polynomial terms.
Since $f$ and $g$ are orthogonal in $\mathcal{H}(T, A, \mu)$,
we have $\sum_n c_n \bar d_n =0$. Hence, by Lemma \ref{verd}, 
$$
\frac{S_1(z)\tilde S_2(z)}{A(z)}  = o\Big(\frac{1}{|z|}\Big),
\qquad
\frac{\tilde S_1(z) S_2(z)}{A(z)}  = o\Big(\frac{1}{|z|}\Big), \qquad 
|z|\to \infty, \ \ z\notin \Omega_1,
$$
where $\Omega_1$ is a set of zero density. On the other hand, by the same 
Lemma \ref{verd}, 
$$
\bigg| \frac{S_1(z) S_2(z)}{A(z)}\bigg|   \gtrsim \frac{1}{|z|},
\qquad
\bigg|  \frac{\tilde S_1(z) \tilde S_2(z)}{A(z)} \bigg|   
\gtrsim \frac{1}{|z|}, \qquad 
z\notin \Omega_2,
$$
for a set $\Omega_2$ of zero density. Note that, e.g.,  
$\frac{S_1(z) S_2(z)}{A(z)}$ has even a larger estimate from 
below if $R$ is a nonzero polynomial. Combined together, these estimates obviously 
lead to a contradiction. 
\medskip
\\
{\bf Step 3.}          
In what follows we will denote by $\mathcal{C}$ 
the set of all functions of the form $\sum_n \frac{d_n}{z-t_n}$, 
where $(d_n) \in \ell^1$.
                                                                            
Assume that ${\rm dim}\, \big(\mathcal{K}(\Lambda_1, \Lambda_2) \big)^\perp > (N+1)^2$
and choose in $\big(\mathcal{K}(\Lambda_1, \Lambda_2) \big)^\perp$
an orthogonal system $\{f_j\}_{j=1}^m \cup \{g_k\}_{k=1}^{N+1}$
where $m> N(N+1)$. Let $(S_1^j, S_2^j)$ 
and $(\tilde S_1^k, \tilde S_2^k)$ be the elements of $\mathcal{S}_{12}$
corresponding to $f_j$ and $g_k$, respectively (see Subsection \ref{her1}). If 
$$
f_j(z) = A(z)\sum_n \frac{c_n^j \mu_n^{1/2}}{z-t_n}, \qquad
g_k(z) = A(z)\sum_n \frac{d_n^k \mu_n^{1/2}}{z-t_n},
$$
then there exist polynomials $U_{jk}$ of degree at most $N-1$ 
such that 
$$
\frac{S_1^j(z)\tilde S_2^k(z)}{A(z)} = \sum_n \frac{c_n^j \bar d_n^k}{z-t_n}
+ U_{jk}(z).
$$
Since $m> (N+1){\rm dim}\, \mathcal{P}_{N-1}$, there exist $\{\alpha_j\}_{j=1}^m$
such that $\sum_{j=1}^m \alpha_j U_{jk} \equiv 0$ for any $k=1, \dots N+1$.
Put $f= \sum_{j=1}^m \alpha_j f_j$ and let  $(S_1, S_2)$ be the corresponding element 
of $\mathcal{S}_{12}$. Then we have $\frac{S_1\tilde S_2^k}{A} \in \mathcal{C}$
for any $k$. On the other hand, there exist polynomials 
$V_k \in \mathcal{P}_{N-1}$ such that 
$$
\frac{\tilde S_1^k S_2}{A} - V_k \in \mathcal{C}.
$$
Choose $\{\beta_k\}_{k=1}^{N+1}$ such that $\sum_{k=1}^{N+1} \beta_k V_k \equiv 1$
and put $g=\sum_{k=1}^{N+1} \beta_k g_k$. If we denote by $(\tilde S_1, 
\tilde S_2)$
the element of $\mathcal{S}_{12}$ corresponding to $g$, then we have
both $\frac{S_1 \tilde S_2}{A} \in \mathcal{C}$ and
$\frac{\tilde S_1 S_2}{A} \in \mathcal{C}$, and, moreover, $(f,g) = 0$. 
As we have already seen in Step 2, this leads to a contradiction.
This shows that
$$
{\rm dim}\, \mathcal{M} \ominus \EE(\MM, \mathcal{T}_G) = 
{\rm dim}\, \big(\mathcal{K}(\Lambda_1, \Lambda_2) \big)^\perp \le (N+1)^2.
$$
Theorem \ref{syn2} is proved. 
\qed
\medskip

\subsection{Proof of Theorem \ref{syn3}}
We will need the following technical lemma. 

\begin{lemma}
\label{pol2}
Let $T=\{t_n\}$ be power separated with power $M$, i.e., $T$
satisfies \eqref{pow}. Let $d_n^{(j)}$, $j=1, \dots 4$, be such that 
$d_n^{(1)} d_n^{(2)} = d_n^{(3)}d_n^{(4)}$ for any $n$ and, 
for some $N\in \mathbb{N}_0$, 
$$
\sum_n \frac{|d_n^{(j)}|}{|t_n|} <\infty, \quad j=1,
\dots, 4,  \qquad \text{and} 
\qquad \sum_n \frac{|d_n^{(j)}|^2}{|t_n|^{2N}\mu_n}<\infty, \quad j=1,3.
$$
Then 
$$
f(z) = A(z) \bigg( \sum_n \frac{d_n^{(1)}}{z-t_n}  \cdot
\sum_n \frac{d_n^{(2)}}{z-t_n}  - 
\sum_n \frac{d_n^{(3)}}{z-t_n}  \cdot
\sum_n \frac{d_n^{(4)}}{z-t_n} \bigg)
$$
belongs to $\mathcal{P}_{M+N} \mathcal{H}(T,A,\mu)$.
\end{lemma}

\begin{proof}
Using formula \eqref{srt1} we can rewrite
$$
\sum_n \frac{d_n^{(j)}}{z-t_n} = P_j(z) + 
z^{M+1} \sum_n \frac{p_n^{(j)}}{z-t_n}, \qquad j=2,4,
$$
where $P_j$ is the polynomial of degree at most $M$
and $p_n^{(j)} = d_n^{(j)} t_n^{-M-1}$ satisfies 
$\sum_n |t_n|^M |p_n^{(j)}| <\infty$. By Lemma \ref{pol1}, 
$$
A(z)P_2(z)\sum_n \frac{d_n^{(1)}}{z-t_n}, \  \ \
A(z)P_4(z)\sum_n \frac{d_n^{(3)}}{z-t_n} \ \in \
\mathcal{P}_{M+N} \mathcal{H}(T,A,\mu).
$$
It remains to show that
$$
A(z) \bigg( \sum_n \frac{d_n^{(1)}}{z-t_n}  \cdot
\sum_n \frac{p_n^{(2)}}{z-t_n}  - 
\sum_n \frac{d_n^{(3)}}{z-t_n}  \cdot
\sum_n \frac{p_n^{(4)}}{z-t_n} \bigg) \in
\mathcal{P}_{N} \mathcal{H}(T,A,\mu).
$$
Note that by condition $d_n^{(1)} d_n^{(2)} = d_n^{(3)}d_n^{(4)}$
the coefficient at $(z-t_n)^{-2}$ is zero. We have
$$
\sum_n \frac{ d_n^{(1)} }{z-t_n}  \cdot
\sum_m \frac{ p_m^{(2)} }{z-t_m} = 
\frac{ d_n^{(1)} p_n^{(2)} } {(z-t_n)^2}  +
\sum_n \bigg(\sum_{m\ne n} 
\frac{ p_m^{(2)} }{t_n-t_m}\bigg) \frac{ d_n^{(1)} }{z-t_n}.
$$
It follows from power separation that $|t_n-t_m|\gtrsim |t_m|^{-M}$, $n\ne m$, and so 
$\big| \sum_{m\ne n} 
\frac{p_m^{(2)}}{t_n-t_m}\big| \lesssim 1$. Now the inclusion
$$
A(z) \sum_n \bigg(\sum_{m\ne n} 
\frac{ p_m^{(2)} }{ t_n-t_m } \bigg) \frac{ d_n^{(1)} }{z-t_n}
\in \mathcal{P}_{N} \mathcal{H}(T,A,\mu) 
$$
follows from Lemma \ref{pol1}.
\end{proof}
\medskip
\begin{proof}[Proof of Theorem \ref{syn3}]
By the symmetry, we can prove the spectral synthesis up to a finite defect
for the operator $\LL^*$ in place of $\LL$. So we interchange the roles 
of $a$ and $b$ and assume that $\LL^*$ is complete,
$$
\sum_n \frac{|a_n|^2}{|b_n|^2|t_n|^{2N}} <\infty, \qquad |b_n|^2 \nu_n 
\gtrsim |t_n|^{-2N-2}.
$$
We show that the adjoint model operator $\mathcal{T}_G^*$ in $\mathcal{H}(T, A, \mu)$
admits spectral synthesis up to a finite defect. 
Assume that $f(z) = A(z)\sum_n \frac{c_n\mu_n^{1/2}}{z-t_n} 
\in \mathcal{H}(T, A, \mu)$ is orthogonal to 
the mixed system $\mathcal{K}(\Lambda_1, \Lambda_2)$ defined by \eqref{mix},
and let $(S_1, S_2) \in \mathcal{S}_{12}$ be the corresponding entire functions 
for which \eqref{para2} holds. Multiplying the equations in \eqref{para2} we get
equality \eqref{mup} which can be rewritten as follows:
$$
G(z) R(z) = 
A(z)\bigg( \sum_n \frac{G(t_n) \bar c_n}{A'(t_n)\mu_n^{1/2} (z-t_n)}
\cdot \sum_n \frac{c_n\mu_n^{1/2}}{z-t_n}  - 
\frac{G(z)}{A(z)} \sum_n \frac{|c_n|^2}{z-t_n}\bigg). 
$$
Note that 
$$
\frac{G(z)}{A(z)} = 1+ z \sum_n \frac{a_n \bar b_n t_n \nu_n}{z- t_n},
$$
Put 
$$
d_n^{(1)} = \frac{G(t_n) \bar c_n}{A'(t_n)\mu_n^{1/2}}, \quad
d_n^{(2)} = c_n \mu_n^{1/2}, \quad
d_n^{(3)} = a_n \bar b_n t_n \nu_n,
\quad    d_n^{(4)} = |c_n|^2.
$$                         
Then $d_n^{(1)}$, $d_n^{(2)}$, $d_n^{(3)}$ and $d_n^{(4)}$  satisfy  
conditions of Lemma \ref{pol2}. The fact that 
$\sum_n |d_n^{(1)}|^2 |t_n|^{-2N} \mu_n^{-1} <\infty$
follows from \eqref{d1}--\eqref{d2} in the 
proof of Theorem \ref{bio2}, while it is clear that 
$$
\sum_n |d_n^{(3)}|^2\mu_n^{-1} = \sum_n |a_n|^2\nu_n <\infty.
$$
Now we may write
$$
G(z)R(z) = -A(z)\sum_n \frac{|c_n|^2}{z-t_n} +
 A(z) \bigg( \sum_n \frac{d_n^{(1)}}{z-t_n}  \cdot
\sum_n \frac{d_n^{(2)}}{z-t_n}  - 
\sum_n \frac{d_n^{(3)}}{z-t_n}  \cdot
\sum_n \frac{d_n^{(4)}}{z-t_n} \bigg).
$$
By Lemma \ref{pol1}, $A(z)\sum_n \frac{|c_n|^2}{z-t_n} 
\in \mathcal{P}_{N} \mathcal{H}(T,A,\mu)$. Indeed, 
$$
\sum_n \frac{|c_n|^4}{|t_n|^{2N}\mu_n}  = 
\sum_n \frac{|c_n|^4}{|t_n|^{2N+2}|b_n|^2\nu_n} \lesssim 
\sum_n |c_n|^4 <\infty
$$        
by \eqref{dom22}.
Hence, by Lemma \ref{pol2}, $GR \in \mathcal{P}_{M+N} \mathcal{H}(T,A,\mu)$.

Recall that $\mathcal{T}_G^*$ is complete and so there is no nonzero function 
of the form $GU \in \mathcal{H}(T,A,\mu)$. If $R$ has at least $M+N$ zeros
counting multiplicities, then, dividing $R$ by a polynomial $P$ of degree
$M+N$ we have $GRP^{-1} \in \mathcal{H}(T,A,\mu)$, a contradiction. 
Thus, $R=P e^Q$, where $P$ is a polynomial
of degree at most $M+N-1$ and $Q$ is some polynomial
(recall that $R$ is of finite order). 
Assume that $Q\ne const$. Since  
$G\in  {\rm Assoc}\, (T,A,\mu)$,  
the function $G(z)\frac{R(z) - R(w)}{z-w}$
belongs to $\mathcal{P}_{M+N} \mathcal{H}(T,A,\mu)$ and has 
infinitely many zeros
for all $w$ except at most 1, a contradiction. Thus, we conclude that 
$Q\equiv const$ and  so $R$ is a polynomial of degree at most $M+N-1$.

The rest of the proof is the same as the proof of Theorem \ref{syn2}.
If we assume that 
${\rm dim}\, \big(\mathcal{K}(\Lambda_1, \Lambda_2) \big)^\perp > (M+N+1)^2$,
then there exist two mutually orthogonal functions $f, g$ in
$\big(\mathcal{K}(\Lambda_1, \Lambda_2) \big)^\perp$ with the  properties
as in Step 2 of the proof of Theorem \ref{syn2} which again leads to a 
contradiction.
\end{proof}
\bigskip


\section{Counterexamples}
\label{th26}

In this  section we prove Theorem \ref{coun}. The following observation is trivial,
but leads to a substantial simplification of the construction compared 
to \cite{bbb1, by}. It says that it is sufficient to construct counterexamples
on  an arbitrarily sparse part of the spectrum. 

\begin{lemma}
\label{bro1}
Let $\A$ be a compact normal cyclic operator
and let $\A = \A_1\oplus \A_2$ with respect to 
decomposition $H = H_1\oplus H_2$. 
Let $K\in \mathbb{N} \cup \{\infty\}$.
Assume that either there exists a rank one perturbation 
$\LL_1$ of $\A_1$ such that $\LL_1$ is complete 
and ${\rm dim}\, (\EE(\LL_1^*))^\perp = K$, 
or $\LL_1$ and $\LL^*_1$ are complete, but 
${\rm dim}\, \mathcal{M}_1 \ominus \EE(\MM_1, \LL_1) = K$ 
for some $\LL_1$-invariant subspace $\MM_1$. Then for 
the operator $\LL = \LL_1 \oplus \A_2$ 
we have, respectively, that $\LL$ is complete, but 
${\rm dim}\, (\EE(\LL^*))^\perp = K$, 
or $\LL$ and $\LL^*$ are complete, but 
${\rm dim}\, \mathcal{M} \ominus \EE(\MM, \LL) = K$ 
for some $\LL$-invariant subspace $\MM$.
\end{lemma}

\begin{proof}
The proof is obvious. Consider, e.g., the statement about
synthesis. Assume that there exists a 
$\LL_1$-invariant subspace $\MM_1$ such that
${\rm dim}\, (\mathcal{M}_1 \ominus \EE(\MM_1, \LL_1)) = K$.
Then $\MM = \mathcal{M}_1\oplus \{0\}$ is $\LL$-invariant 
and $\EE(\MM, \LL) = \EE(\MM_1, \LL_1) \oplus \{0\}$.  
\end{proof}

It is clear that in Lemma \ref{bro1} 
$\LL$ is a rank one perturbation of $\A$. Thus we see that
it is sufficient to construct examples for the restriction of 
$\A$ to any invariant (with respect to $\A$ and $\A^*$) subspace.
We will choose $H_1 = L^2(\nu_1)$ where $\nu_1$ is the restriction of the initial measure 
$\mu$ to an infinite, but sparse part of the spectrum $\{s_n\}$.
By sparseness we will mean (Hadamard-type) {\it lacunarity} 
of the inverse spectrum $\{t_n\}$,
$t_n = s_n^{-1}$: 
\begin{equation}
\label{lac0}
\inf_n \bigg|\frac{t_{n+1}}{t_n}\bigg| >1.
\end{equation}

In view of Lemma \ref{bro1} in our examples below we can always
assume that $\A$ is a compact normal operator with spectrum $\{s_n\}$
such that $\{t_n\}$ is a lacunary sequence.
\medskip


\subsection{Proof of Theorem \ref{coun}: biorthogonal systems 
with finite defect.} 
\label{fin}
Let $N\in \mathbb{N}$. We will prove the following statement:
\medskip

{\it Let $T=\{t_n\}$ be a lacunary sequence satisfying \eqref{lac0}. 
Then there exists a space 
$\mathcal{H}(T, A, \mu)$ and a function 
$G \in {\rm Assoc}\,(T, A, \mu) \setminus \mathcal{H}(T, A, \mu)$ 
with simple zeros such that
for $\Lambda = \mathcal{Z}_G$ the system $\{k_\lambda\}_{\lambda\in \Lambda}$
is complete in $\mathcal{H}(T, A, \mu)$, but 
$$
{\rm dim}\, \bigg(\mathcal{H}(T, A, \mu)\ominus
\ospan\Big\{\frac{G}{z-\lambda}:\lambda\in \Lambda\Big\}\bigg) = N. 
$$ } 

Thus, the  model operator $\mathcal{T}_G$ is incomplete with defect $N$, 
while its biorthogonal is complete. 
Then $\LL = \mathcal{T}_G^*$, the adjoint to the model operator, 
will be a rank one perturbation of a compact normal operator 
with spectrum $\{\bar s_n\}$, $s_n = t_n^{-1}$,
and its adjoint will be incomplete with defect $N$. 

Let $A(z) = \prod\big(1-\frac{1}{t_n}\big)$. 
Put $\mu_n = |t_n|^{-2N}$. Let $\tilde t_n = t_n + \frac{1}{2}$, 
$\tilde T = \{\tilde t_n\}$
and $\tilde A(z) = \prod (1-z/\tilde t_n)$. 
Then, by the standard estimates of infinite products with lacunary zeros, 
for $z\in \mathbb{C}\setminus T$,
\begin{equation}
\label{lac}
\bigg|\frac{\tilde A(z)}{A(z)}\bigg|
\asymp \bigg| \frac{{\rm dist}\, (z, \tilde T)}{{\rm dist}\, (z, T)} \bigg|, \qquad
\bigg|\frac{\tilde A(t_n)}{A'(t_n)}\bigg|\asymp 1, \quad t_n\in T.
\end{equation}

We will also use the following simple observation:  if $\sum_n |d_n| <\infty$, then
\begin{equation}
\label{lac1}
\bigg| \sum_n\frac{d_n}{z-t_n} \bigg| = o(1), \quad
\bigg| \sum_n\frac{t_n d_n}{z-t_n} \bigg| = o(|z|), \qquad 
|z|\to \infty, \ \ {\rm dist}\, (z, T) \ge 1.
\end{equation}
In particular, $|f(z)| = o(|zA(z)|)$ for any $f\in 
\mathcal{H}(T, A, \mu)$ when  
$|z|\to \infty$, ${\rm dist}\, (z, T) \ge 1$.

Let $P$ be a polynomial of degree $N$ such that $\mathcal{Z}_P
\subset \mathcal{Z}_{\tilde A}$. Put $G = \tilde A/P$, $\Lambda = 
\mathcal{Z}_G$. 
\medskip
\\
{\bf Step 1:} $G\in {\rm Assoc}\,(T, A, \mu) \setminus \mathcal{H}(T, A, \mu)$.
\\
We have
$$
\frac{|G(t_n)|^2}{|A'(t_n)|^2\mu_n} = 
\frac{|\tilde A(t_n)|^2}{|A'(t_n)|^2|P(t_n)|^2\mu_n} \asymp 1,
$$
whence $G\notin \mathcal{H}(T, A, \mu)$ by Lemma \ref{gr2}. On the other hand,
$\sum_n \frac{|G(t_n)|^2}{|A'(t_n)|^2|t_n|^2\mu_n} <\infty$. Let us show
that $\frac{G}{z-\lambda} \in \mathcal{H}(T, A, \mu)$ for any $\lambda\in \Lambda$. 
We need to show that
$$
\frac{G(z)}{(z-\lambda)A(z)} =
\sum_n \frac{G(t_n)}{(t_n-\lambda)A'(t_n) (z-t_n)}.
$$
The following standard argument will be repeated several times in the 
constructions below. Put 
$$
H(z) = \frac{G(z)}{(z-\lambda)A(z)} - 
\sum_n \frac{G(t_n)}{(t_n-\lambda)A'(t_n) (z-t_n)}.
$$
The poles cancel and so $H$ is an entire function. 
We have $|G(z)|\asymp |P(z)|^{-1}|A(z)|$, ${\rm dist}\, (z, T) \ge 1$.
Also, by \eqref{lac1}, the last series is $o(1)$ when $|z|\to \infty$
and ${\rm dist}\, (z, T) \ge 1$. Hence, $|H(z)|\to 0$, $|z|\to\infty$, and so 
$H\equiv 0$. We conclude that $\frac{G}{z-\lambda} \in \mathcal{H}(T, A, \mu)$.
\medskip
\\
{\bf Step 2:} {\it $\{k_\lambda\}_{\lambda\in \Lambda}$ 
is complete in $\mathcal{H}(T, A, \mu)$.}
\\
Assume that $GU \in \mathcal{H}(T, A, \mu)$ for some entire $U$.
Then, $G(z)U(z) = A(z)\sum_n \frac{d_n \mu_n^{1/2}}{z-t-n}$
for some $(d_n)\in\ell^2$ and so, by \eqref{lac1}, 
$|G(z)U(z)/A(z)| = o(|z|)$ when $|z|\to \infty$, 
${\rm dist}\, (z, T) \ge 1$. Also,
$$
\bigg| \frac{G(z)U(z)}{A(z)} \bigg| 
= \bigg| \frac{\tilde A(z)U(z)}{A(z)P(z)}\bigg| 
\asymp \bigg|\frac{U(z)}{P(z)} \bigg|,
\qquad {\rm dist}\, (z, T) \ge 1. 
$$
We conclude that $U$ is a polynomial. 

Since $GU \in \mathcal{H}(T, A, \mu)$ we have
$$
\sum_n |U(t_n)|^2 \asymp 
\sum_n \frac{|G(t_n)U(t_n)|^2}{|A'(t_n)|^2\mu_n} <\infty,
$$
and so $U\equiv 0$.
\medskip
\\
{\bf Step 3:} 
${\rm dim}\, \big\{ \frac{G(z)}{z-\lambda}:\lambda\in \Lambda\big\}^\perp = N$. 
\\
We use parametrization of the orthogonal complement to a system
of the form $\big\{\frac{G}{z-\lambda}\big\}_{\lambda\in \Lambda}$ 
introduced in Subsection \ref{param}. It is parametrized 
by the space $\mathcal{S}$ which consists of entire functions $S$ such that
\begin{equation}
\label{biox}
S(z)G(z) = A(z) \sum_n \frac{\bar c_n G(t_n)}{A'(t_n)\mu_n^{1/2}(z-t_n)}
\end{equation}
where $(c_n) \in \ell^2$. Assume first that $S\in \mathcal{S}$. Then it 
follows from representation \eqref{biox} that 
$S(z)G(z) = A(z) \sum_n \frac{d_n}{z-t_n}$, where $\sum_n |t_n|^{-1}|d_n| <\infty$, 
and so, by \eqref{lac1}, 
$$
\bigg| \frac{S(z)\tilde A(z)}{A(z)P(z)} \bigg| = 
\bigg| \frac{S(z)G(z)}{A(z)} \bigg| = o(|z|), \qquad
{\rm dist}\, (z, T) \ge 1.
$$
It follows from the estimates \eqref{lac} that $S$ is a polynomial.
Also we have
$c_n = \overline{S(t_n)} \mu_n^{1/2}$, and so $\sum_n |S(t_n)|^2\mu_n<\infty$. 
Since $\mu_n = |t_n|^{-2N}$ we conclude that degree of $S$ does not exceed $N-1$.

Conversely, let $S\in \mathcal{P}_{N-1}$ and put 
$c_n = \overline{S(t_n)} \mu_n^{1/2}$. We need to show that \eqref{biox} holds, 
that is, there is the interpolation formula
\begin{equation}
\label{biox1}
\frac{S(z)G(z)}{A(z)} = \sum_n \frac{S(t_n)G(t_n)}{A'(t_n)(z-t_n)}.
\end{equation}
We argue as in Step 1. The residues at the points $t_n$
coincide and so the difference between the left-hand side
and the right-hand side of \eqref{biox1}
is an entire function.  We have
$$
\bigg|\frac{S(z)G(z)}{A(z)}\bigg| \lesssim\frac{1}{|z|}, 
\qquad \bigg| 
\sum_n \frac{S(t_n)G(t_n)}{A'(t_n)(z-t_n)}\bigg| =o(1), \qquad
|z|\to\infty, \ \ {\rm dist}\, (z, T) \ge 1, 
$$
and so \eqref{biox1} holds. Thus, $\mathcal{S} = \mathcal{P}_{N-1}$ 
and 
${\rm dim}\, \big\{ \frac{G}{z-\lambda}:\lambda\in \Lambda\big\}^\perp 
={\rm dim}\, \mathcal{S} = N$.
\qed
\medskip


\subsection{Proof of Theorem \ref{coun}: biorthogonal systems 
with infinite defect.} 
\label{infin}
As in Subsection \ref{fin}, for a lacunary sequence $T$, we construct
a measure $\mu$ and a function 
$G \in {\rm Assoc}\,(T, A, \mu) \setminus \mathcal{H}(T, A, \mu)$ such that
for $\Lambda = \mathcal{Z}_G$ the system $\{k_\lambda\}_{\lambda\in \Lambda}$
is complete in $\mathcal{H}(T, A, \mu)$, but 
$$
{\rm dim}\, \bigg(\mathcal{H}(T, A, \mu)\ominus
\ospan\Big\{\frac{G}{z-\lambda}:\lambda\in \Lambda\Big\}\bigg) = \infty. 
$$  

The following lemma from \cite{bbb1} will be crucial here. 
One can also use a simpler, but less sharp result of \cite[Lemma 7.1]{by}.
As usual, for a sequence $\Gamma$, we denote
by $n_\Gamma$ its counting function: 
$n_\Gamma(r) = \#\{\gamma \in \Gamma: |\gamma| <r\}$. For an entire function $f$ we write $n_f$
in place of $n_{\mathcal{Z}(f)}$. 

\begin{lemma} \textup(\cite[Lemma 9.2]{bbb1}\textup)
\label{newl}
Let $\Gamma$ be a lacunary sequence and let $f$
be an entire function of zero exponential type such that 
$$
\int_0^R\frac{n_f(r)}{r}dr = o\bigg(\int_0^R  \frac{n_\Gamma(r)}{r}dr\bigg),
\qquad  R\to \infty.
$$ 
If $\{f(\gamma)\}_{\gamma\in \Gamma}\in \ell^\infty$, then $f$ is a constant. 
\end{lemma}
\medskip
{\bf Step 1:} {\it Construction of $G$. }

Choose an infinite $T_0 \subset T$ such that $n_{T_0} = o(n_{T\setminus T_0})$.
As in the previous subsection let  
$\tilde T = \{\tilde t_n\}$, $\tilde t_n = t_n + \frac{1}{2}$, 
and let 
$$
\tilde A(z) = \prod_n\bigg(1-\frac{1}{\tilde t_n}\bigg), \qquad
U(z) = \prod_{t_n \in T_0} \bigg(1-\frac{1}{\tilde t_n}\bigg)
$$
Define the measure $\mu=\sum_n \mu_n \delta_{t_n}$ by
$$
\mu_n =  
\begin{cases}
1, & t_n\in T_0, \\
|U'(t_n)|^{-2}, & t_n\in T\setminus T_0.
\end{cases}
$$
Finally, put $G=\tilde A/U$. Then we have
$$
\begin{aligned}
\frac{G(z)}{A(z)} & = \sum_n \frac{\tilde A(t_n)}{A'(t_n)U(t_n) (z-t_n)}, \\
\frac{G(z)}{(z-\lambda) A(z)} & = \sum_n \frac{\tilde A(t_n)}{A'(t_n)U(t_n) 
(t_n-\lambda) (z-t_n)}, \qquad \lambda\in \Lambda = \mathcal{Z}_G. 
\end{aligned}
$$
Convergence of the above series follows from estimate \eqref{lac} 
and the fact that $|U(t_n)|\gtrsim |t_n|^N$ for any $N>0$. The interpolation formulas 
follow by the same arguments as in Subsection \ref{fin}, Step 1. 
We also have
$$
\sum_{n} \frac{|G(t_n)|^2}{|A'(t_n)|^2 \mu_n} \asymp 
\sum_{n} \frac{1}{|U(t_n)|^2 \mu_n} = \infty,
$$
since $|U(t_n)|^2 \mu_n = 1$, $t_n\in T\setminus T_0$. However, 
for any $\lambda\in \Lambda=\mathcal{Z}_G$,
$$
\sum_{n} \frac{|G(t_n)|^2}{|t_n-\lambda|^2|A'(t_n)|^2 \mu_n} \asymp 
\sum_{t_n\in T_0} \frac{1}{|t_n-\lambda|^2 |U(t_n)|^2} +
\sum_{t_n\in T\setminus T_0} \frac{1}{|t_n-\lambda|^2} < \infty.
$$
Thus, $G\notin \mathcal{H}(T, A, \mu)$, but $\frac{G(z)}{z-\lambda} 
\in \mathcal{H}(T, A, \mu)$, $\lambda\in \Lambda$.
\medskip
\\
{\bf Step 2:} {\it $\{k_\lambda\}_{\lambda\in \Lambda}$ 
is complete in $\mathcal{H}(T, A, \mu)$.}

Assume that the system $\{k_\lambda\}_{\lambda\in \Lambda}$  is not complete
and so there is an entire function $V$ such that $GV \in  \mathcal{H}(T, A, \mu)$.
We have, by \eqref{lac1}, 
$$
\bigg|\frac{V(z)}{U(z)} \bigg| =
\bigg|\frac{G(z)V(z)}{\tilde A(z)} \bigg| \asymp
\bigg|\frac{G(z)V(z)}{A(z)} \bigg| \lesssim |z|, \qquad  
{\rm dist}\, (z, T) \ge 1.  
$$
Hence, $\log M_V(r) \lesssim \log M_U(r)$ 
(where $M_f(r) = \max_{|z| = r}|f(z)|$) and so, by the classical Iensen formula, 
$$
\int_0^R\frac{n_V(r)}{r}dr \lesssim 
\int_0^R  \frac{n_{T_0} (r)}{r}dr.
$$
On the other hand, 
$$
\sum_n \frac{|V(t_n)|^2}{|U(t_n)|^2 \mu_n} \asymp
\sum_n \frac{|G(t_n)V(t_n)|^2}{|A'(t_n)|^2 \mu_n} <\infty,
$$
whence $\sum_{t_n \in T\setminus T_0} |V(t_n)|^2 <\infty$. 
Since $n_{T_0} = o(n_{T\setminus T_0})$, $V\equiv 0$ by Lemma \ref{newl}.
\medskip
\\
{\bf Step 3:} 
${\rm dim}\, \big\{ \frac{G(z)}{z-\lambda}:\lambda\in \Lambda\big\}^\perp = \infty$. 

Construct an entire function
$$
S(z) = \prod_{t_n\in T_0} \bigg(1- \frac{z}{t_n +\vep_n} \bigg),
$$
where $\vep_n\ne 0$ are chosen to be so small that $\sum_{t_n\in T_0} |S(t_n)|^2 <\infty$.
Then, similarly to \eqref{lac},  we have $|S(t_n)| \asymp |U(t_n)|$, $t_n\in
T\setminus T_0$. Let $P$ be an arbitrary nonconstant polynomial such that 
$\mathcal{Z}_P \subset \mathcal{Z}_S$. Then
$$
\sum_n \bigg|\frac{G(t_n)S(t_n)}{A'(t_n)P(t_n)}\bigg| <\infty
$$
and 
$$
\sum_n \bigg|\frac{S(t_n)}{P(t_n)}\bigg|^2 \mu_n 
= \sum_{t_n \in T_0} \bigg|\frac{S(t_n)}{P(t_n)}\bigg|^2 
+ \sum_{t_n \in T\setminus T_0} \frac{|S(t_n)|^2}{|P(t_n)|^2 
|U(t_n)|^2} <\infty.
$$
Put $c_n = \overline{S(t_n)} \mu_n^{1/2}/\overline{P(t_n)}$. Then,
by the argument used in Subsection \ref{fin}, Step 1, we have the interpolation formula
$$
\frac{G(z)S(z)}{P(z)A(z)} = \sum_n 
\frac{G(t_n)S(t_n)}{A'(t_n)P(t_n)(z-t_n)} = 
\sum_n \frac{G(t_n)\bar c_n}{A'(t_n)\mu_n^{1/2} (z-t_n)}.
$$ 
Hence, by Subsection \ref{param}, $f(z) = A(z)\sum_n \frac{c_n \mu_n^{1/2}}{z-t_n}$
is orthogonal to $\big\{\frac{G(z)}{z-\lambda}\big\}_{\lambda\in \Lambda}$
and $S/P$ belongs to the corresponding space $\mathcal{S}$. 
Choose a sequence $P_j$ of polynomials such that
$P_j$ is a polynomials of degree $j$ and $\mathcal{Z}_{P_j} \subset \mathcal{Z}_S$.
Clearly, the system $S/P_j$ is linearly independent in $\mathcal{S}$ whence
${\rm dim}\, \big\{ \frac{G(z)}{z-\lambda}:\lambda\in \Lambda\big\}^\perp = \infty$. 
\medskip


\subsection{Proof of Theorem \ref{coun}: complete perturbations without synthesis.}
Let $N\in \mathbb{N} \cup\{\infty\}$. By the results of Subsections \ref{fin} 
and \ref{infin}, for any lacunary spectrum, there is an example of a 
complete system of reproducing kernels whose biorthogonal system
is incomplete with defect $N$. 

Now let $T$ be lacunary and let $T = T_1 \cup T_2$ so that $n_{T_2}(r) = 
o(n_{T_1}(r))$, $r\to \infty$. Let $A_2(z) = \prod_{t_n\in T_2} (1-z/t_n)$.
Assume that we have chosen a measure 
$\mu^{(2)} = \sum_{t_n \in T_2} \mu_n^{(2)} \delta_{t_n}$ and a function 
$G_2$ such that
\begin{itemize}
\item
$G_2$ is a canonical product with lacunary zeros;
\smallskip
\item
$G_2 \in {\rm Assoc}\,(T_2, A_2, \mu^{(2)}) \setminus \mathcal{H}(T_2, A_2, \mu^{(2)})$;
\smallskip
\item
for $\Lambda_2 = \mathcal{Z}_{G_2}$ the system 
$\{k^{(2)}_\lambda\}_{\lambda\in \Lambda_2}$
is complete in $\mathcal{H}(T_2, A_2, \mu^{(2)})$;
\smallskip
\item ${\rm dim}\, \big(\mathcal{H}(T_2, A_2, \mu^{(2)})\ominus
\ospan\big\{\frac{G_2(z)}{z-\lambda}:\lambda\in \Lambda_2\big\}\big) = N.$ 
\end{itemize}
\smallskip
Here $k^{(2)}_\lambda$ denote the reproducing kernel
in the space $\mathcal{H}(T_2, A_2, \mu^{(2)})$.
Thus, there exist $N$ linearly independent functions of the form 
$f(z) = A_2(z)\sum_{t_n\in T_2} \frac{c_n (\mu^{(2)}_n)^{1/2}}{z-t_n}$ such that 
\begin{equation}
\label{bab2}
A_2(z) \sum_{t_n\in T_2} \frac{G_2(t_n) \bar c_n }{A_2'(t_n)(\mu^{(2)}_n)^{1/2} (z-t_n)}  
= G_2(z)S_2(z)
\end{equation}                                    
for some entire function $S_2$. 

Put $\mu_n = 1$, $t_n\in T_1$, and $\mu_n = \mu_n^{(2)}$, $t_n \in T_2$,
and consider the corresponding Cauchy--de~Branges space 
$\mathcal{H}(T, A, \mu)$, $A=A_1A_2$.
Define $c_n = 0$, $t_n\in T_1$. Then, 
multiplying the formula for $f$ and \eqref{bab2} by $A_1$, we get
$$
\begin{aligned}
A_1(z) f(z) & = A_1(z)A_2(z) \sum_{t_n\in T} \frac{c_n \mu_n^{1/2}}{z-t_n},
\\
G_2(z) A_1(z) S_2(z) & = 
A_1(z)A_2(z) \sum_{t_n \in T}
\frac{G_2(t_n) A_1(t_n) \bar c_n }{(A_1A_2)'(t_n)\mu_n^{1/2} (z-t_n)}. 
\end{aligned}
$$
By the discussion in Subsection \ref{her1}, these equations are equivalent to the fact             
that the function $A_1f$ is orthogonal to the mixed system
\begin{equation}
\label{bab3}
\{k_\lambda\}_{\lambda\in T_1} \cup
\Big\{\frac{A_1(z) G_2(z)}{z-\lambda}\Big\}_{\lambda\in \Lambda_2}
\end{equation}
in $\mathcal{H}(T, A, \mu)$ (here we denote by $k_\lambda$ the reproducing kernel 
of the space $\mathcal{H}(T, A, \mu)$. Moreover, if some function 
$g \in \mathcal{H}(T, A, \mu)$ is orthogonal to the system \eqref{bab3}, then 
$g(t_n) =0$, $t_n \in T_1$. Hence, $g=A_1f$ where
$f(z)=A_2(z)\sum_{t_n \in T_2}\frac{c_n \mu_n^{1/2}}{z-t_n} 
\in \mathcal{H}(T_2, A_2, \mu^{(2)})$. Writing the equation 
for the orthogonality of $A_1 f$ to
$\big\{\frac{A_1(z) G_2(z)}{z-\lambda}\big\}_{\lambda\in \Lambda_2}$,
we get \eqref{bab2} with some entire $S_2$ and so 
$f$ is orthogonal to $\big\{\frac{G_2(z)}{z-\lambda}\big\}_{\lambda\in \Lambda_2}$
in $\mathcal{H}(T_2, A_2, \mu^{(2)})$. Thus, the codimension of the system 
\eqref{bab3} in $\mathcal{H}(T, A, \mu)$ equals the codimension of the system
$\big\{\frac{G_2(z)}{z-\lambda}\big\}_{\lambda\in \Lambda_2}$
in $\mathcal{H}(T_2, A_2, \mu^{(2)})$, that is, $N$. 
\smallskip

Let us show that the system $\{k_\lambda\}_{\lambda\in T_1\cup\Lambda_2}$
is complete in $\mathcal{H}(T, A, \mu)$. Assume that $A_1G_2V \in 
\mathcal{H}(T, A, \mu)$ for some entire function $V$. Then 
there exists a sequence $(d_n)\in \ell^2$ such that
$$
A_1(z)G_2(z)V(z) = A_1(z)A_2(z) \sum_n \frac{d_n \mu_n^{1/2}}{z-t_n},
$$
and so $d_n = 0$, $t_n \in T_1$. Hence, 
$$                                                                  
G_2(z)V(z) = A_2(z) \sum_{t_n\in T_2} 
\frac{d_n \mu_n^{1/2}}{z-t_n} \in \mathcal{H}(T_2, A_2, \mu^{(2)}),
$$
a contradiction to completeness of 
$\{k^{(2)}_\lambda\}_{\lambda\in \Lambda_2}$ in $\mathcal{H}(T_2, A_2, \mu^{(2)})$.
\smallskip

It remains to show that the biorthogonal system 
$\big\{\frac{A_1(z) G_2(z)}{z-\lambda}\big\}_{\lambda\in T_1 \cup \Lambda_2}$
is complete in $\mathcal{H}(T, A, \mu)$. 
Assume that there exists $(c_n)\in \ell^2$ such that 
the function $g(z) = A(z)\sum_n \frac{c_n \mu_n^{1/2}}{z-t_n}$ 
is orthogonal to the above system. Then, as in Subsection \ref{param},
we have, for $\lambda \in \Lambda_2$, 
$$
0= \Big(\frac{A_1G_2}{z-\lambda}, g\Big) = 
\sum_n \frac{A_1(t_n) G_2(t_n) \bar c_n}{A'(t_n)\mu_n^{1/2}(t_n-\lambda)} = 
\sum_{t_n\in T_2} \frac{G_2(t_n) \bar c_n}{A_2'(t_n)\mu_n^{1/2}(t_n-\lambda)},
$$
and so there exists an entire function $S_2$ such that 
$$
A_2 (z)\sum_{t_n\in T_2} \frac{G_2(t_n) \bar c_n}{A_2'(t_n)\mu_n^{1/2}(z- t_n)} 
= G_2(z)S_2(z).
$$
On the other hand, if $\lambda = t_m \in T_1$, then
$$
\begin{aligned}
0= \Big(\frac{A_1 G_2}{z-t_m}, g\Big) & = 
\sum_{t_n\in T_2} \frac{A_1(t_n) G_2(t_n) \bar c_n}{A'(t_n)\mu_n^{1/2}(t_n-t_m)}
+ \frac{A_1'(t_m) G_2(t_m)\bar c_m}{A'(t_m)\mu_m^{1/2}} \\
& = 
\sum_{t_n\in T_2} \frac{G_2(t_n) \bar c_n}{A_2'(t_n)\mu_n^{1/2}(t_n-t_m)}
+ \frac{G_2(t_m)\bar c_m}{A_2(t_m)\mu_m^{1/2}}.
\end{aligned}
$$
We conclude that 
$$
\frac{G_2(t_m)\bar c_m}{A_2(t_m)\mu_m^{1/2}} = 
\frac{G_2(t_m) S_2(t_m)}{A_2(t_m)}, 
\qquad t_m \in T_1,
$$
and so $S_2(t_m) = \bar c_m$ (recall that $\mu_m =1$).  

By estimates \eqref{lac1}, $|G_2(z)S_2(z)/A_2(z)| \lesssim |z|$, 
${\rm dist}\, (z, T_2) \ge 1$. Since $G_2$ is lacunary product
we conclude that 
$\log M_{S_2}(r) \lesssim \log M_{A_2}(r)$, whence
$$
\int_0^R\frac{n_{S_2}(r)}{r}dr \lesssim 
\int_0^R  \frac{n_{T_2} (r)}{r}dr.
$$
At the same time 
$n_{T_2}(r) = o(n_{T_1}(r))$, $r\to \infty$, and $\{S_2(t_n)\}_{t_n\in T_1}
\in \ell^2$. Hence, by Lemma \ref{newl},
$S_2 \equiv 0$ and so the system
$\big\{\frac{A_1(z) G_2(z)}{z-\lambda}\big\}_{\lambda\in T_1 \cup \Lambda_2}$
is complete. This completes the proof of Theorem \ref{coun}.
\qed
\bigskip


\section{Proof of ordering theorems}
\label{th27} 

\subsection{Nearly invariant and division-invariant subspaces.}
Let $\mathcal{H}$ be a reproducing kernel Hilbert space of functions
analytic in some domain $D$. A closed subspace $\mathcal{H}_0$ of 
$\mathcal{H}$ is said to be {\it nearly invariant} if there is 
$w_0\in D$ such that $\frac{f(z)}{z-w_0} \in \mathcal{H}_0$ 
whenever $f\in \mathcal{H}_0$
and $f(w_0) =0$. This notion goes back to the work of Hitt \cite{hitt}
and Sarason \cite{sar}. Usually it is assumed that $w_0 = 0$. 
It is known that nearly invariance is equivalent
to a stronger {\it division invariance property}. 
Denote by $\mathcal{Z}(\mathcal{H}_0)$ 
the set of common zeros for $\mathcal{H}_0$,
that is, the set of $w\in D$ such that $f(w) =0 $ for any $f\in
\mathcal{H}_0$. Then the division invariance means that 
for any $w\in D \setminus \mathcal{Z}(\mathcal{H}_0)$,
$$
f\in \mathcal{H}_0, \ \ f(w) = 0 \ \Longrightarrow \ 
\frac{f(z)}{z-w}\in \mathcal{H}_0.
$$
In the context of Hardy spaces in general domains the equivalence
of nearly invariance and division invariance is shown 
in \cite[Proposition 5.1]{ar}; a similar argument works 
for general spaces of analytic functions.

\begin{proposition}
\label{nea1}
Let $\mathcal{H}_0$ be a nearly invariant subspace 
of some reproducing kernel Hilbert space $\mathcal{H}$ of functions
analytic in some domain $D$.
Then, for any $w\in D\setminus \mathcal{Z}(\mathcal{H}_0)$
and any $f\in \mathcal{H}_0$ such that $f(w) =0$, we have
$\frac{f(z)}{z-w} \in \mathcal{H}_0$.
\end{proposition}

\begin{proof}
We show that the set of $w$ in $D\setminus \mathcal{Z}(\mathcal{H}_0)$ satisfying
the conclusions of the proposition is both open and closed in 
$ D\setminus \mathcal{Z}(\mathcal{H}_0)$.  

First of all note that, for any $w\in D\setminus \mathcal{Z}(\mathcal{H}_0)$, 
the operator $f\mapsto \frac{f(z)}{z-w}$ is bounded 
from the subspace $\{f\in \mathcal{H}:\, f(w) = 0\}$ to $\mathcal{H}$
by the Closed Graph Theorem. 
Let $f(w) = 0$ and write
$\frac{f(z)}{z-w} = g_w +h_w$, where $g_w \in \mathcal{H}_0$,
$h_w \perp \mathcal{H}_0$. 
Fix some function $f_0\in \mathcal{H}_0$ such that $f_0(w_0) = 1$.
Then, by a simple computation, 
$$
\frac{f - f(w_0)f_0}{z-w_0}  = g_w+h_w + (w_0-w)\bigg(
\frac{g_w - g_w(w_0)f_0}{z-w_0} + \frac{h_w - h_w(w_0)f_0}{z-w_0}
\bigg) \ \in \mathcal{H}_0.
$$
Since $g_w$ and $\frac{g_w - g_w(w_0)f_0}{z-w_0}$ are in $\mathcal{H}_0$,
we conclude that
$$
h_w = (w-w_0) {\rm P}_{\mathcal{H}_0^\perp}
\bigg(\frac{h_w - h_w(w_0)f_0}{z-w_0}\bigg),
$$
where ${\rm P}_{\mathcal{H}_0^\perp}$ denotes the orthogonal projection. 
This is impossible when $w$ is sufficiently close to $w_0$ unless
$h_w =0$, since $\big\|\frac{h_w - h_w(w_0)f_0}{z-w_0}
\big\| \le C \|h_w\|$ for some constant $C$ independent  on $w$.
We used the fact that $\mathcal{H}$ is a reproducing kernel Hilbert space 
to estimate $h_w(w_0)$ by $\|h_w\|$. Thus, we have seen that if we can 
divide in $\mathcal{H}_0$ by $z-w_0$ then we can
divide by $z-w$ with $w$ close to $w_0$.

Let us show that the set of $w$ satisfying the conclusion of the proposition
is closed. Let $w \in D\setminus \mathcal{Z}(\mathcal{H}_0)$ and assume that
$w_n\to w$ is a sequence such that one can divide by $z-w_n$ 
in $\mathcal{H}_0$. Fix some function $g \in \mathcal{H}_0$ such that 
$g(w) \ne 0$. We may assume that $g(w_n)\ne 0$ for all $n$.
Then all operators $T_{w_n}$, where
$$
T_{w}f = \frac{f- \frac{f(w)}{g(w)} g}{z-w},
$$   
are bounded. An easy computation leads to a resolvent-type identity 
$$
T_{w_n} - T_w = (w_n-w) T_{w_n}T_w.
$$
Therefore $\|T_{w_n}\| \le \|T_w\| + |w_n-w|\cdot \|T_{w_n}\|\cdot\|T_w\|$,
whence $\sup_n \|T_{w_n}\|<\infty$. It follows that 
$\|T_{w_n} - T_w\|\to 0$, $n\to \infty$. Since $T_{w_n}f\in 
\mathcal{H}_0$ for any $f\in \mathcal{H}_0$, we conclude that
$T_w \mathcal{H}_0 \subset \mathcal{H}_0$.                                          
\end{proof}
\medskip

\subsection{Reduction to an ordering theorems for nearly invariant subspaces.}
We show that Theorems \ref{ord1} and \ref{ord2} are equivalent 
to an ordering theorem for nearly invariant subspaces.
Let $\LL = \A +a\otimes b$ be a rank one perturbation
of a compact normal operator with simple spectrum $\{s_n\}$. 
Passing to the model operator 
$\mathcal{T}_G$ in the Cauchy--de
Branges space $\mathcal{H}(T, A, \mu)$, $T=\{t_n\} = \{s_n\}^{-1}$,
the problem becomes 
equivalent to the following. 
Let $\MM$ be an invariant subspace of $\mathcal{T}_G$. Put 
$\Lambda = \mathcal{Z}_G$ and let $\Lambda_2$ be the set of those
$\lambda\in \Lambda$ for which $\frac{G}{z-\lambda} \in \MM$. 
Then
\begin{equation}
\label{zaz}
\ospan \Big\{\frac{G}{z-\lambda}: 
\ \lambda\in \Lambda_2 \Big\} \subset \MM \subset 
\big( \ospan \{k_\lambda:\ \lambda\in \Lambda_1\} \big)^\perp,
\end{equation}
where $\Lambda_1 = \Lambda\setminus \Lambda_2$.
Since $\frac{G}{z-\lambda}\in \MM$ if and only if 
$\lambda\in \Lambda_2$, we conclude that the set of common zeros 
$\mathcal{Z}(\MM)$ coincides with $\Lambda_1$.

Recall that $G(0)=1$ and so $0\notin \Lambda$. 
Since $\MM$ is $\mathcal{T}_G$-invariant, we have, in particular,
$$
F\in \MM, \ F(0)  =0 \ \Longrightarrow \ \frac{F(z)}{z}\in \MM,
$$
that is, $\MM$ is nearly invariant and, hence, division-invariant 
by Lemma \ref{nea1}.  
                     
Now  Theorems \ref{ord1} and \ref{ord2} follow from the following 
ordering theorem.

\begin{theorem}
\label{ord3}
Let $\mathcal{H} = \mathcal{H}(T, A, \mu)$ be a de Branges--Cauchy space
and assume  that one of the folowing conditions 
holds\textup:
\begin{enumerate}
\item[(i)]
$T\subset \mathbb{R}$ and $|t_n|\to \infty$ as $|n|\to \infty$\textup;
\item[(ii)]
$T$ satisfies one of the conditions 
$\mathbf{Z}$, $\mathbf{\Pi}$ or $\mathbf{A_\gamma}$.
\end{enumerate}
Let $G\in {\rm Assoc}\,(T, A, \mu)$ and $\Lambda = \mathcal{Z}_{G} $. 
Assume that $\{k_\lambda\}_{\lambda\in \Lambda}$ is a complete 
and minimal system
of reproducing kernels in $\mathcal{H}$ and let
$\Lambda = \Lambda_1 \cup \Lambda_2$, $\Lambda_1 \cap \Lambda_2 = \emptyset$. 
Then the set of all nearly invariant subspaces $\MM$ satisfying \eqref{zaz}
is totally ordered by inclusion.
\end{theorem}

The proof of this theorem is based on the ideas of L. de Branges 
and is very similar to the proof of \cite[Theorem 35]{br}
or \cite[Theorems 1.3, 1.4]{abb}. However, our hypothesis is somewhat 
different and we cannot just refer to the above results.
Therefore, and in order to make the exposition more self-contained, 
we present below the proof of Theorem \ref{ord3}.
\smallskip 


\subsection{Preliminaries on the Smirnov class and Krein's theorem.}
\label{kre} 
Here we will use some basic notions of Hardy spaces theory 
and Nevanlinna inner-outer factorization (see, e.g., \cite{ko}).
Recall that a function $f$ analytic in $\CC^+$ is said to be of 
{\it bounded type} in $\CC^+$,
if $f=g/h$ for some functions $g$, $h$ analytic and bounded in $\CC^+$. 
If, moreover, $h$ can be taken to be outer, we say that $f$
belongs to the {\it Smirnov class} in $\cp$.  

Assume now that $T\subset\cm$ and consider a discrete Cauchy transform 
\begin{equation}
\label{kaa}
f(z) = \sum_n \frac{c_n}{t_n-z}, \qquad \sum_n \frac{|c_n|}{|t_n|}<\infty.
\end{equation}
Then $f$ is in the Smirnov class in $\cp$. This follows from the fact
that $\ima f>0$ in $\cp$ if $c_n>0$. Any function with positive imaginary part 
belongs to the Smirnov class, as well as a sum of two Smirnov class functions.

In particular, a discrete Cauchy transform 
with real poles is a function of Smirnov class 
both in the upper half-plane $\mathbb{C}^+$
and in the lower half-plane $\mathbb{C}^-$.  Therefore, 
in the case $T\subset \mathbb{R}$, the function $f/A$ is of 
Smirnov class in $\CC^+$ and $\CC^-$  
for any $f\in \mathcal{H}(T, A, \mu)$ and, thus, zeros 
of $f$ satisfy the Blaschke condition in $\CC^+$ and $\CC^-$. 

A classical result by M.G. Krein says that if $f$ is an entire function
which is in the Smirnov class both in $\CC^+$ and in $\CC^-$, then
$f$ is of zero exponential type.
For different approaches to this result see 
\cite[Part II, Chapter 1]{hj},  \cite[Lecture 16]{lev} or \cite{br}.

In what follows we will need the following (very standard)
observation that the generating function of a complete and minimal 
system must have a maximal growth with respect to order 1.

\begin{lemma} 
\label{maxi}
Let $\mathcal{H} = \mathcal{H}(T, A, \mu)$ be a de Branges--Cauchy space
such that $T \subset \{-r\le \ima z\le r\}$ and 
either $r=0$ \textup(i.e., $T\subset \mathbb{R}$\textup) or $r>0$ and $T$
has finite convergence exponent.
Let $G\in {\rm Assoc}\,(T, A, \mu)$ and $\Lambda = \mathcal{Z}_{G}$. 
Assume that the system $\{k_\lambda\}_{\lambda\in \Lambda}$ is complete 
and minimal in $\mathcal{H}$. Then the inner-outer factorizations
for $G/A$ in $\cp+ir$ and $\cm-ir$ 
are, respectively, $G/A = \mathcal{O} B$ and $G/A = \widetilde{\mathcal{O}} 
\widetilde{B}$, where
$\mathcal{O}, \widetilde{\mathcal{O}}$ are the corresponding outer functions 
and $B, \widetilde{B}$ are
some Blaschke products.                                
\end{lemma}

\begin{proof}
We prove the factorization in $\cp+ir$, the case of $\cm -ir$
is analogous. Since $\frac{G}{z-\lambda} \in \mathcal{H}$
for any $\lambda\in \Lambda$, we conclude that $\frac{G}{A(z-\lambda)}$ 
is a discrete Cauchy transform of the form \eqref{kaa}
and so is in the Smirnov class in $\cp$. 
Hence, $G/A = \mathcal{O} Be^{iaz}$ with $a\ge 0$
(since $G/A$ is meromorphic in $\mathbb{C}$, the singular inner factor
reduces to $e^{iaz}$). Assume that $a>0$. Put 
\begin{equation}
\label{sl1}
H(z) = \frac{G(z)(e^{-iaz/2}-1)}{z A(z)} - 
\sum_n  \frac{G(t_n)(e^{-iat_n/2} -1)}{t_n A'(t_n)(z-t_n)}.
\end{equation}
Since $G\in {\rm Assoc}\,(T, A, \mu)$ and the sequence $\{e^{-iat_n/2}\}$ 
is bounded,
the above Cauchy transform converges absolutely. The residues at $t_n$
coincide and so $H$ is an entire function.

Consider first the case $r=0$. Then $G/A$ is in the Smirnov class in 
$\cp$ and $\cm$. Moreover, since $a>0$, the function 
$Ge^{-iaz/2}/A = 
\mathcal{O} Be^{iaz/2}$ also is in the Smirnov class in $\cp$ and $\cm$. 
Since the discrete 
Cauchy transform in \eqref{sl1} also is in the Smirnov class
in $\cp$ and in $\cm$, 
we conclude that $H$ is of zero exponential type by Krein's theorem. 

Recall that any Smirnov class function $u$ in $\cp$ satisfies
\begin{equation}
\label{smi}
\log^+ |u(r e^{i\theta})| = o(r), \qquad r\to \infty,
\end{equation} 
for any fixed $\theta \in (0, \pi)$ (as usual, $\log^+ t = \max (\log t, 0)$). 
Hence, $Ge^{-iaz/2}/A$ tends to zero 
along the imaginary axis when $y\to \pm \infty$. Thus, $|H(iy)|\to 0$,
$|y|\to \pm\infty$, and we conclude that  $H\equiv 0$.
Now we have 
$$
\frac{G(z)(e^{-iaz/2}-1)}{z} = A(z)\sum_n  
\frac{G(t_n)(e^{-iat_n/2} -1)}{t_n A'(t_n)(z-t_n)},
$$
whence $\frac{G(e^{-iaz/2}-1)}{z} \in \mathcal{H}$, 
a contradiction to the fact that the system $\{k_\lambda\}_{\lambda\in 
\Lambda}$ is complete in $\mathcal{H}$.

If $r>0$, we argue similarly to show that $Ge^{-iaz/2}/A$ (and, hence, $H$) 
is of Smirnov class in $\cp+ir$  and $\cm-ir$. Since we also know that
$A$ is of finite order, we conclude by Lemma \ref{gr1} that $H$ is 
of finite order. Combining this with \eqref{smi} we see that $H$ is 
of zero exponential type by the standard Phragm\'en--Lindel\"of 
principle and, hence, $H\equiv 0$. The end of the proof 
is the same as for $r=0$. 
\end{proof}


\subsection{Proof of Theorem \ref{ord3}.}
In what follows we put 
$$
\nu = \sum_n \frac{\delta_{t_n}}{|A'(t_n)|^2\mu_n}.
$$
By Lemma \ref{gr2},  the embedding $\mathcal{H}(T,A,\mu)$ 
into $L^2(\nu)$ is isometric and
so we can compute the norm and inner product in
$\mathcal{H}(T,A,\mu)$ as the integral with respect to $\nu$. 

Assume that $\MM_1$ and $\MM_2$ are two subspaces satisfying 
\eqref{zaz} and neither $\MM_1 \subset \MM_2$ nor
$\MM_2 \subset \MM_1$. Choose nonzero functions
$F_1, F_2 \in \mathcal{H}$ such that $F_1 \perp \MM_2$ but
$F_1$ is not orthogonal to  $\MM_1$, while
$F_2 \perp \MM_1$ but $F_2$ is not orthogonal to $\MM_2$.

We fix two functions $G_1$ and $G_2$ such that 
$G=G_1G_2$, $\Lambda_1 = \mathcal{Z}_{G_1} $, 
$\Lambda_2 = \mathcal{Z}_{G_2}$.
Note that if $\MM$ satisfies \eqref{zaz}, then any function  
in $\MM$ is of the form $G_1F$ for some entire $F$.

Now let $G_1 F \in \MM_1$ and $G_1 H \in \MM_2$. 
Define two functions
$$
\begin{aligned}
f(w) & = \bigg\langle G_1\frac{F - 
\frac{F(w)}{H(w)} H}{z-w}, 
F_1\bigg\rangle_{\mathcal{H}(T,A,\mu)} = 
\int G_1(z) 
\frac{F(z) - \frac{F(w)}{H(w)} H(z)}{z-w} \overline{F_1(z)} 
d\nu(z), \\
h(w) & = \bigg\langle G_1 \frac{H - \frac{H(w)}{F(w)}F}{z-w}, 
F_2\bigg\rangle_{\mathcal{H}(T,A,\mu)}
= \int G_1(z)\frac{H(z) - \frac{H(w)}{F(w)}F(z)}{z-w}
\overline{F_2(z)} d\nu(z).
\end{aligned}
$$
The functions $f$ and $h$ are well-defined and analytic 
on the sets $\{w: H(w)\ne 0\}$
and $\{w: F(w)\ne 0\}$, respectively.
\medskip
\\
{\bf Step 1:} {\it $f$ and $h$ are entire functions, 
$f$ does not depend on the choice of $H$ and 
$h$ does not depend on the choice of $F$}.

Let $\tilde f$ be a function associated in a similar way to 
another function $G_1\tilde H \in \MM_2$,
$$
\tilde f(w)  = 
\int G_1(z) \frac{F(z) - \frac{F(w)}{\tilde H(w)}\tilde H(z)}{z-w} 
\overline{F_1(z)} d\nu(z).
$$
Then, for $\tilde F(w) \ne 0$ and $\tilde H(w) \ne 0$, we have
$$
\tilde f(w) - f(w) = \frac{F(w)}{H(w)\tilde H(w)} 
\int G_1(z)\frac{\tilde H(w)H(z) -H(w)\tilde H(z)}{z-w} 
\overline{F_1(z)} d\nu(z) = 0,
$$
since 
\begin{equation}
\label{vk}
G_1 \frac{\tilde H(w) H - H(w)\tilde H}{z-w}
\in \MM_2. 
\end{equation}
This inclusion is obvious when $w\notin\Lambda_1$ since $\MM_2$ 
is division-invariant; since the norm in $\mathcal{H}(T,A,\mu)$
coincides with the norm in $L^2(\nu)$, by continuity we obtain
the inclusion \eqref{vk} for $w\in \Lambda_1$ as well.

For any $w$ we can choose $H$ such that $H(w) \ne 0$
and so we can extend $f$ analytically to any point $w$.
Thus, $f$ and $h$ are entire functions.
\medskip
\\
{\bf Step 2:} {\it $f$ and $h$ are of zero exponential type.}
\smallskip
\\
{\bf Case $T\subset \mathbb{R}$.} We have
$$     
f(w) = \int G_1(z) 
\frac{F(z)\overline{F_1(z)}}{z-w}  d\nu(z) - 
\frac{F(w)}{H(w)} 
\int G_1(z) \frac{H(z)\overline{F_1(z)}}{z-w}  d\nu(z).
$$
Since $G_1F, G_1H \in 
\mathcal{H}(T, A, \mu)$
and the points $t_n$ are real, the functions $G_1F/A$ \
and $G_1H/A$ are in the Smirnov class in $\cp$
and $\cm$ by discussion in Subsection \ref{kre}. 
The function $f$ does not depend on the choice 
of $H$, and we can take $H = \frac{G_2}{z-\lambda}$ for some $\lambda
\in \Lambda_1$. Then, by Lemma \ref{maxi}, $G_1H/A$ has no exponential 
factor in its canonical factorization in $\cp$ and we conclude 
that $F/H = u/B$ for some Smirnov class function $u$ and some Blaschke product
$B$ in $\cp$. The discrete Cauchy transforms 
$$
\int \frac{G_1(z)F(z)\overline{F_1(z)}}{z-w}  d\nu(z), \qquad  
\int \frac{G_1(z)H(z)\overline{F_1(z)}}{z-w}  d\nu(z),
$$ 
are also in Smirnov class in $\cp$ (as functions of $w$), 
and so $f$ is in the Smirnov class
in $\cp$. Similarly, $f$ is in the Smirnov class in $\cm$. By Krein's theorem
$f$ is of zero exponential type.
\smallskip
\\
{\bf Case $\mathbf{\Pi}$.} We assume without loss of generality that
$T\subset \{-r\le \ima z \le r\}$.
Similarly to the above case we have that $f$ is of the Smirnov class in 
$\cp+ir$ and $\cm - ir$. By the hypothesis 
$T$ has a finite convergence exponent and so $A$ has finite order.
By Lemma \ref{gr1} the functions $G_1F$, $G_1H$, 
$$
A(w)\int \frac{G_1(z)F(z)\overline{F_1(z)}}{z-w}  d\nu(z), \qquad  
A(w)\int \frac{G_1(z)H(z)\overline{F_1(z)}}{z-w}  d\nu(z),
$$
are all of finite order, whence $f$ is of finite order. 
Since any function $u$ in the Smirnov class in $\cp$ satisfies
$\log^+ |u(r e^{i\theta})| = o(r)$, $r\to \infty$, 
for any $\theta \in (0, \pi)$, the 
standard Phragm\'en--Lindel\"of principle implies that $f$ is 
of zero exponential type.                                              
\smallskip
\\
{\bf Case $\mathbf{Z}$ or $\mathbf{A_\gamma}$.}
Here we can refer to \cite[Theorem 1.4, Corollary 3.1]{abb}: If $T$ satisfies
$\mathbf{Z}$ or $\mathbf{A_\gamma}$ and an entire function $f$ 
satisfies $f = u_1 + u_2\frac{u_3}{u_4}$ where $u_j$ are discrete
Cauchy transforms with poles in $T$, then $f$ is of zero exponential type.

Similarly, $h$ is of zero exponential type.  
\medskip
\\
{\bf Step 3:} {\it Either $f$ or $h$ is identically zero.}

Given $w$ such that $F(w)\ne 0$, $H(w) \ne 0$, 
we have
\begin{equation}
\label{ots}
\begin{aligned}
|f(w)| & \le \bigg|\int \frac{G_1(z)F(z)\overline{F_1(z)}}{z-w} d\nu(z) \bigg| + 
\bigg|\frac{F(w)}{H(w)}\bigg|\cdot 
\bigg|\int \frac{G_1(z)H(z)\overline{F_1(z)}}{z-w} d\nu(z) \bigg|,\\
|h(w)| & \le \bigg|\int \frac{G_1(z)H(z)\overline{F_2(z)}}{z-w} d\nu(z) \bigg| + 
\bigg|\frac{H(w)}{F(w)}\bigg|\cdot 
\bigg|\int \frac{G_1(z)F(z)\overline{F_2(z)}}{z-w} d\nu(z)\bigg|.
\end{aligned}
\end{equation}
If $T\subset \mathbb{R}$, then by a rough estimate of the Cauchy transforms,  
we have 
$$        
|f(w)| \lesssim \frac{1}{|\ima w|}\bigg(1+ \bigg|\frac{F(w)}{H(w)}\bigg|\bigg), \qquad
|h(w)| \lesssim \frac{1}{|\ima w|} \bigg(1+ \bigg|\frac{H(w)}{F(w)}\bigg|
\bigg). 
$$
From this we conclude that
$$
\min\big(|f(w)|, |h(w)|\big) \lesssim \frac{1}{|\ima w|}.
$$
Applying a well-known and  deep result by de Branges 
\cite[Lemma 8]{br}, we conclude that either $f$ or $h$ is zero. 

In the cases $\mathbf{Z}$, $\mathbf{\Pi}$ or $\mathbf{A_\gamma}$ we
argue as in the proof of \cite[Theorem 1.3]{abb}.
To estimate the Cauchy trnsforms in \eqref{ots} we use Lemma \ref{gr1}:
there exists $M>0$ and a set $E\subset (0, \infty)$ of zero 
linear density such that
$$         
|f(w)| \lesssim |w|^M\bigg(1+ \bigg|\frac{F(w)}{H(w)}\bigg|\bigg), \qquad
|h(w)| \lesssim |w|^M\bigg(1+ \bigg|\frac{H(w)}{F(w)}\bigg|\bigg), 
\qquad |w|\notin E.
$$
We conclude that
$$
\min\big(|f(w)|, |h(w)|\big) \lesssim |w|^M, \qquad |w|\notin E.
$$
Since $E$ has zero linear density, we can choose a sequence 
$R_j\to \infty$ such that $R_j\notin E$ and $R_{j+1}/ R_j \le 2$. Applying 
the maximum principle to the annuli $R_j\le |z|\le R_{j+1}$, we conclude that
$$
\min\big(|f(w)|, |h(w)|\big) \lesssim |w|^M, \qquad |w|\ge 1.
$$
Since both $f$ and $h$ are of zero exponential type, 
a small variation of \cite[Lemma 8]{br} gives that either $f$ 
or $h$ is a polynomial. 

Assume that $f$ is a nonzero polynomial. 
Now we define a line $\Gamma$ which is separated from $T$ (at infinity). 
If $T\subset \{-r\le \ima z\le r\}$, put $\Gamma=i\mathbb{R}$.
If $T$ is contained in some angle 
of the size $\pi\gamma$, $\gamma<1$ (say, $\{0\le \arg z\le \pi\gamma\}$),
put $\Gamma = \{\rho e^{i(\pi \gamma + \delta)}:\ \rho\in \mathbb{R}\}$,
where $0<\delta<\pi(1-\gamma)$. In each of these cases,
by Lemma \ref{ugol}, we have 
\begin{equation} 
\label{sl2}
\bigg|\int\frac{G_1(z) F(z)\overline{F_1(z)}}{z-w} d\nu(z) \bigg|
+
\bigg|\int\frac{G_1(z) H(z)\overline{F_1(z)}}{z-w} d\nu(z) \bigg| 
=O\Big(\frac{1}{|w|}\Big), 
\end{equation}
when $|w|\to\infty$, $w\in\Gamma$. 
Hence, from \eqref{ots}, 
$|F(w)/H(w)| \to\infty$ as $|w|\to\infty$, $w\in\Gamma$, and so
$$
\begin{aligned}
|h(w)| & \le \bigg|\int\frac{G_1(z) H(z)\overline{F_2(z)}}{z-w} d\nu(z) \bigg| 
\\ & \qquad \qquad + 
\bigg|\frac{H(w)}{F(w)}\bigg|\cdot 
\bigg|\int \frac{G_1(z) F(z)\overline{F_2(z)}}{z-w} d\nu(z)\bigg|  
= O\Big(\frac{1}{|w|}\Big), \qquad w\in \Gamma. 
\end{aligned}
$$
Using the fact that $h$ is of zero exponential type, we conclude that $h\equiv 0$.

In the case $\mathbf{Z}$ we have no information about location of the points
$t_n$. However, by Lemma \ref{verd}, there exists a set 
$\Omega$ of zero area density 
such that  \eqref{sl2} holds when $w\notin \Omega$.
Hence, $|F(w)/H(w)| \to\infty$ as $|w|\to\infty$, $w\notin\Omega$, and so
$$   
\begin{aligned}
|h(w)| & \le \bigg|\int\frac{G_1(z) H(z)\overline{F_2(z)}}{z-w} d\nu(z) \bigg| 
\\
& \qquad \qquad + 
\bigg|\frac{H(w)}{F(w)}\bigg|\cdot 
\bigg|\int \frac{G_1(z) F(z)\overline{F_2(z)}}{z-w} d\nu(z)\bigg| 
= O\Big(\frac{1}{|w|}\Big), \qquad w\notin \Omega\cup \widetilde{\Omega}, 
\end{aligned}
$$
where $\widetilde{\Omega}$ is another set of zero area density 
(here we again applied Lemma \ref{verd}). 
Thus, $h$ tends to zero outside a set of zero density
and so $h\equiv 0$ by Theorem \ref{dens}.
\medskip
\\
{\bf Step 4:} {\it End of the proof.} 

Without loss of generality, let $f\equiv 0$. Then
\begin{equation} 
\label{sl6}
\frac{F(w)}{H(w)}\int\frac{G_1(z) H(z)\overline{F_1(z)}}{z-w}  d\nu(z)
=\int\frac{G_1(z) F(z)\overline{F_1(z)}}{z-w}  d\nu(z)
\end{equation}                                            
for any $G_1 F\in \mathcal{M}_1$, $G_1 H\in \mathcal{M}_2$.

Recall that $F_1$ is not orthogonal to $\MM_1$ and so we can choose
$G_1 F \in \MM_1$ such that
$\langle G_1F, F_1\rangle_{\mathcal{H}(T, A, \mu)} 
= \int G_1F\overline{F}_1 d\nu \ne 0$.
Assume first that
$T\subset \{-r\le \ima z\le r\}$ or $T$ satisfies $\mathbf{A_\gamma}$
and let $\Gamma$ be the line constructed in Step 3.
We now compare the asymptotics of the right-hand side and on the 
left-hand side of \eqref{sl6} on $\Gamma$. 
By Lemma \ref{ugol}, we have
\begin{equation} 
\label{sl3}
\bigg|\int\frac{G_1(z) F(z)\overline{F_1(z)}}{z-w}  d\nu(z)\bigg| 
\gtrsim \frac{1}{|w|}, 
\end{equation} 
when $|w|\to \infty$, $w\in \Gamma$. Since $G_1H\perp F_1$ for 
any $G_1 H\in \MM_2$, we have (again by Lemma \ref{ugol}) 
\begin{equation} 
\label{sl4}
\bigg|\int\frac{G_1(z)H(z)\overline{F_1(z)}}{z-w}  
d\nu(z)\bigg| = o\Big(\frac{1}{|w|}\Big), 
\end{equation} 
when $|w|\to\infty$, $w\in \Gamma$.
We conclude that $|F(w)/H(w)| \to \infty$ when $|w|\to \infty$,
$w\in \Gamma$. Applying this fact and 
Lemma \ref{ugol} to $h$ we conclude from \eqref{ots} 
that $|h(w)|\to 0$ 
when $|w|\to \infty$, $w\in \Gamma$. Since $h$ is of zero exponential type,
we have $h\equiv 0$.

Thus, we have
\begin{equation}
\label{ghj}
\frac{H(w)}{F(w)}\int\frac{G_1(z)F(z)\overline{F_2(z)}}{z-w}  d\nu(z)
=\int\frac{G_1(z) H(z)\overline{F_2(z)}}{z-w}  d\nu(z)
\end{equation}
and we may repeat the above argument. Choose $G_1H \in \MM_2$
such that $\langle G_1H, F_2\rangle_{\mathcal{H}(T, A, \mu)} 
= \int G_1H\overline{F}_2 d\nu \ne 0$. 
Then, by Lemma \ref{ugol}, the modulus of the 
right-hand side in \eqref{ghj} is $\gtrsim |w|^{-1}$,
while the left-hand side is $o(|w|^{-1})$ when $|w|\to\infty$, $w\in \Gamma$. 
This contradiction proves Theorem \ref{ord3} in the cases when 
$T\subset \{-r\le \ima z\le r\}$ (in particular, $T\subset\mathbb{R}$) 
or $T$ satisfies $\mathbf{A_\gamma}$.

The proof for the case $\mathbf{Z}$ is similar but instead of
asymptotics along a line we consider the asymptotics outside
a set of zero density. Let $F$ be chosen as above. 
By Lemma \ref{verd},
we have \eqref{sl3} and \eqref{sl4} when $|w|\to\infty$, $w\notin\Omega$,  
for some set $\Omega$ of zero density.  
Applying this fact and 
Lemma \ref{verd} to $h$ we conclude that $|h(w)|\to 0$ 
outside a set of zero density
and so $h\equiv 0$ by Theorem \ref{dens}. Then we obtain \eqref{ghj}
and, repeating the argument once again, 
find that the modulus of the 
right-hand side in \eqref{ghj} is $\gtrsim |w|^{-1}$,
while the left-hand side is $o(|w|^{-1})$ when $|w|\to\infty$ 
outside a set of zero density. Case $\mathbf{Z}$
of Theorem \ref{ord3} is also proved.
\qed
\medskip


\section{Volterra rank one perturbations}
\label{volt} 

In this section we discuss the following problem: 
{\it for which compact normal operators $\A$ there exists a rank one perturbation
$\LL$ which is a Volterra operator}\,? In the case when $\A$ is a compact 
selfadjoint operator this question was answered in \cite{by0} in terms of
the so-called Krein class of entire functions. 

Recall that an entire function $F$ 
is said to be in the {\it Krein class} if 
\begin{itemize}
\item $F$ is real on $\RR$ and has simple real zeros $t_n$
(we assume for simplicity that $t_n \ne 0$);
\item for some integer $k\ge 0$, we have
$\sum_n \dfrac{1}{|t_n|^{k+1} |F'(t_n)|}<\infty$;
\item there exists a polynomial $R$ such that
\begin{equation}
\label{krein}
\frac{1}{F(z)} = R(z) + \sum_n
\frac{1}{F'(t_n)} \cdot \bigg(\frac{1}{z-t_n} +\frac{1}{t_n}+
\cdots + \frac{z^{k-1}}{t_n^k} \bigg).
\end{equation}
\end{itemize}
This class was introduced by M.G.~Krein who showed
that in this case $F$ is necessarily of the Cartwright class and, in particular,
of exponential type (see \cite[Theorem 5]{krein47} 
or \cite[Lecture 16]{lev}).
Indeed, $F$ is of bounded type in $\cp$ and $\cm$, 
whence (by yet another result of Krein, see Subsection \ref{kre})
$F$ is of exponential type.
Some further refinements of this result are due to
A.G.~Bakan and V.B.~Sherstyukov (see, e.g., \cite{shers}
and references 
\medskip therein).

One can extend the definition of the Krein class to the case when 
$T=\{t_n\} \subset \mathbb{C}$ no longer is assumed to be real. 
If $F$ satisfies all other conditions above, 
we say that $F$ belongs to the {\it generalized Krein class}. However,
one cannot extend Krein's theorem to this situation: 
a function in the generalized Krein class can have arbitrarily large order. 
Indeed, for a Krein class 
function $F$ one can write \eqref{krein} with $z^m$ in place of $z$, 
where $m\in\mathbb{N}$. 
Then, expanding $(z^m-t_n)^{-1}$ as a sum of simple fractions one can show that
$\tilde F(z) = F(z^m)$ is in the generalized Krein class.

In what follows we will say that an entire function $F$ 
with simple zeros in the set $T=\{t_n\} \subset \mathbb{C}\setminus\{0\}$ 
belongs to the (generalized) Krein class $\mathcal{K}_1$ if 
\begin{equation}
\label{krein1}
\frac{1}{F(z)} = \frac{1}{F(0)} + \sum_n \frac{1}{F'(t_n)}
\bigg(\frac{1}{z-t_n}+\frac{1}{t_n}\bigg), \qquad \sum _n
\frac{1}{|t_n|^2 |F'(t_n)|} <\infty.
\end{equation}
Then the main result of \cite{by0} can be stated as follows:

\begin{theorem} \textup(\cite{by0}\textup)
\label{annih2}
Let $\A$ be a compact selfadjoint operator with simple spectrum $\{s_n\}$, 
$s_n \ne 0$, and let $t_n =s_n^{-1}$. Then the following are equivalent\textup:

{\rm (i)} There exists a rank one perturbation $\LL$ of 
$\A$ which is a Volterra operator\textup;

{\rm (ii)} There exists a function $F \in \mathcal{K}_1$
such that the zero set of $F$ coincides with $\{t_n\}$.
\end{theorem}

In the case when $\A$ satisfies condition (i) above, we say that 
the spectrum $\{s_n\}$ (or $\{t_n\}$) is {\it removable}.
An unexpected (and rather counterintuitive) consequence of Theorem \ref{annih2} 
is that adding a finite number of points to the spectrum helps it to
become removable, while deleting  a finite number of points 
from a removable spectrum may make it nonremovable.

A similar description of spectra removable by a rank one perturbation
holds true for compact normal operators.

\begin{theorem}
\label{annih3}
Let $\A$ be a compact normal operator with simple spectrum $\{s_n\}$, 
$s_n \ne 0$, and let $t_n =s_n^{-1}$. Then the following are equivalent\textup:
\smallskip

{\rm (i)} There exists a rank one perturbation $\LL$ of 
$\A$ which is a Volterra operator\textup;
\smallskip

{\rm (ii)} There exist a Cauchy--de Branges space 
$\mathcal{H}(T, A, \mu)$,  $T= \{t_n\}$, 
and an entire function $G\in {\rm Assoc}\,(T, A, \mu)$  
which does not vanish in $\mathbb{C}$\textup; 
\smallskip 

{\rm (iii)} There exists an entire function $A \in \mathcal{K}_1$
such that the $\mathcal{Z}_A = T$.
\end{theorem}

\begin{proof}
(i)$\Longrightarrow$(ii): Assume that $\LL$ is a Volterra perturbation.
By the functional model of Theorem \ref{main1} there exist
a Cauchy--de Branges space $\mathcal{H}(T, A, \mu)$
and an entire function $G\in {\rm Assoc}\,(T, A, \mu)$, $G(0)=1$,   
such that $\LL$ is unitarily equivalent to $\mathcal{T}_G$. Since 
$\mathcal{T}_G$ is a Volterra operator, we conclude that $G\ne 0$
(otherwise, any $\lambda \in \mathcal{Z}_G$ is an eigenvalue for 
$\mathcal{T}_G$).
\medskip

(ii)$\Longrightarrow$(iii):
Let $G$ be a nonvanishing function in ${\rm Assoc}\,(T, A, \mu)$. Then $G = e^{H}$
for some entire function $H$. It is clear that 
$e^{-H} \mathcal{H}(T, A, \mu) = \mathcal{H}(T, e^{-H}A, \mu)$ 
and $1 = e^{-H} G \in {\rm Assoc}\,(T, e^{-H} A, \mu)$. Therefore, changing $A$
we may assume that $1\in {\rm Assoc}\,(T, A, \mu)$.

Now let $g(z) = A(z)\sum_n \frac{c_n\mu_n^{1/2}}{z-t_n}$ 
be an arbitrary function in $\mathcal{H}(T, A, \mu)$ such that $g(0)=1$.
Then $\frac{1-g(z)}{z} \in \mathcal{H}(T, A, \mu)$ and so 
$$
1 = A(z)\bigg(\sum_n \frac{c_n\mu_n^{1/2}}{z-t_n}  + 
z \sum_n \frac{d_n\mu_n^{1/2}}{z-t_n}\bigg)
$$
for some $(d_n)\in\ell^2$. We conclude that 
$$
\frac{1}{A(z)} = -\sum_n \frac{c_n\mu_n^{1/2}}{t_n} + \sum_n
(c_n + d_nt_n)\mu_n^{1/2}\bigg(\frac{1}{z-t_n} + \frac{1}{t_n}\bigg). 
$$
It follows that $A'(t_n)^{-1} = (c_n + d_nt_n)\mu_n^{1/2}$
and so $A$ satisfies \eqref{krein1}. Thus, $A\in \mathcal{K}_1$.
\medskip

(iii)$\Longrightarrow$(i):
Let $A$ satisfy \eqref{krein1}. Similarly to the above computations,
one can show that $\frac{1-g(z)}{z} \in \mathcal{H}(T, A, \mu)$ 
for any $g\in \mathcal{H}(T, A, \mu)$ such that $g(0)=1$.
Thus, $G\equiv 1 \in {\rm Assoc}\,(T, A, \mu)$. By Theorem \ref{main1}
there exists a rank one perturbation $\LL$ which is unitarily equivalent to 
$\mathcal{T}_G$. Since $G\ne 0$, $\mathcal{T}_G$ obviously is a Volterra operator.
\end{proof}

\begin{remark}
{\rm In view of the role of the Krein class in the description of 
removable spectra, it is a natural problem to extend Krein's theorem 
to generalized Krein class imposing some conditions on $T$.
V.B.~Sherstykov \cite{shers} showed that 
the conclusion of Krein's theorem (i.e., $F$ is of exponential type) 
remains true if $F$ is of finite order, satisfies \eqref{krein} and 
$T$ is contained in some strip. Recently it was shown in \cite{abb}
that if $T$ is contained in some angle of size $\pi\gamma$, $\gamma\in(0,1)$
and $F$ is a function of order strictly less than $1/\gamma$
satisfying \eqref{krein}, then $F$ is of zero exponential type. 
Moreover, this growth restriction is sharp: for any $\gamma\in (0,1)$
there exists $F$ in the generalized Krein class with zeros
in an angle of size $\pi\gamma$ and of order exactly $1/\gamma$. }
\end{remark}

\begin{example}
{\rm It is clear that the class $\mathcal{K}_1$ is stable under multiplication 
by polynomials. Thus, adding a finite set to $T$ does not change the property 
to be  removable by a rank one perturbation. However, deleting 
a finite number of points may turn a removable spectrum into a nonremovable one.
Thus, removable spectra have a certain rigidity.

Let us give some concrete examples (the first two appeared in \cite{by0}):
\medskip

(i) Let $T = \big\{\pi\big(n+\frac{1}{2}\big)\big\}_{n\in\mathbb{Z}}$ and 
$A(z) = \cos\pi z$. 
Then $A\in \mathcal{K}_1$ and $T$ is removable. However, 
$T\setminus\{t_n\}$ is nonremovable  for any $t_n\in T$, since
for $F(z) = A(z)/(z-t_n)$ one has $|F'(t_m)|\asymp |t_m|^{-1}$, $m\ne n$,
and the series $\sum_{m\ne n} |F'(t_m)|^{-1}t_m^{-2}$ diverges.
\medskip

(ii) Let $T = \big\{\pi\big(n+\frac{1}{2}\big)^2\big\}_{n\in\mathbb{N}_0}$ 
and $A(z) = \cos\pi \sqrt{z}$. Then $T$ is removable, but 
it becomes nonremovable after deleting any of its elements.
\medskip

(iii) Let $A(z) = \cos (\pi z^k)$, $k\in \mathbb{N}$. 
Then it is not difficult to show that 
$A\in \mathcal{K}_1$ and so its zero set 
$T= \big\{\big|n+\frac{1}{2}\big|^{1/k}e^{\pi i j/k} \big\}_{n\in \mathbb{N}_0, 0\le j\le 2k-1}$ 
is removable. However, $T\setminus\{t_n\}$ is nonremovable  for any $t_n\in T$. }
\end{example}
\bigskip


\end{document}